 \documentclass[reqno]{amsart}
\usepackage{tikz,genyoungtabtikz,enumitem,
rotating,latexsym,bm,stmaryrd
}
 \usetikzlibrary{positioning,intersections}
  \usepackage{amsmath,amsthm,amsfonts,amssymb
  }
  \usepackage{mathrsfs}
 \usepackage{caption,cases}
\usepackage{lipsum,wasysym}
\usepackage{mathtools}
\usetikzlibrary{arrows,calc,through,backgrounds,matrix,decorations.pathmorphing}

\usetikzlibrary{hobby}    
\usepackage{tikz-cd}
 \usepackage[a4paper,margin=1.15in]{geometry}
\usepackage{cleveref}

\newtheorem{thm}{Theorem}[section]
\newtheorem{cor}[thm]{Corollary}
\newtheorem{conj}[thm]{Conjecture}
\newtheorem{lem}[thm]{Lemma}
\newtheorem{prop}[thm]{Proposition}
\newtheorem*{prop*}{Proposition}
\newtheorem*{thm1}{Theorem A}
\newtheorem*{thm2}{Theorem B}

\newtheorem*{cor*}{Corollary}
\newtheorem*{conj*}{Conjecture}

\newtheorem*{move1'}{Move $\mathbf{1^*}$}

\theoremstyle{remark}

\newtheorem{rmk}[thm]{Remark}

\newtheorem*{Acknowledgements*}{Acknowledgements}

\theoremstyle{definition}
\newtheorem{defn}[thm]{Definition}
\newtheorem{eg}[thm]{Example}

\crefname{defn}{Definition}{Definitions}
\crefname{thm}{Theorem}{Theorems}
\crefname{prop}{Proposition}{Propositions}
\crefname{lem}{Lemma}{Lemmas}
\crefname{cor}{Corollary}{Corollaries}
\crefname{conj}{Conjecture}{Conjectures}
\crefname{section}{Section}{Sections}
\crefname{subsection}{Subsection}{Subsections}
\crefname{eg}{Example}{Examples}
\crefname{figure}{Figure}{Figures}
\crefname{rem}{Remark}{Remarks}
\crefname{rmk}{Remark}{Remarks}
\crefname{equation}{equation}{equation}

\renewcommand{\sharp}{\ell}
\Crefname{defn}{Definition}{Definitions}
\Crefname{thm}{Theorem}{Theorems}
\Crefname{prop}{Proposition}{Propositions}
\Crefname{lem}{Lemma}{Lemmas}
\Crefname{cor}{Corollary}{Corollaries}
\Crefname{conj}{Conjecture}{Conjectures}
\Crefname{section}{Section}{Sections}
\Crefname{subsection}{Subsection}{Subsections}
\Crefname{eg}{Example}{Examples}
\Crefname{figure}{Figure}{Figures}
\Crefname{rem}{Remark}{Remarks}
\Crefname{rmk}{Remark}{Remarks}

\numberwithin{equation}{section}

 \newcommand{\W}{{\mathfrak S}}
\newcommand{\Sym}{{\mathfrak S}}
\newcommand{\elemvec}{{\mathsf e}}

\newcommand{\Hom}{\operatorname{Hom}}

\newcommand{\End}{\operatorname{End}}
\newcommand{\ind}{\operatorname{ind}}

\newcommand{\suchthat}{\;\ifnum\currentgrouptype=16 \middle\fi|\;} 

\newcommand{\NN}{{{\mathbb N}}}

\renewcommand{\le}{\leqslant}
\renewcommand{\ge}{\geqslant}
\newcommand{\GL}{\mathrm{GL}}

\newcommand{\g}{\mathfrak{g}}

\newcommand{\CC}{\mathbb{C}}

\newcommand{\N}{\mathrm{N}}

\newcommand{\C}{\mathbb{C}}

\newcommand{\Pdelta}{P_r(\delta_1\delta_2)}
\newcommand{\Pmn}{P_r(mn)}
\newcommand{\Fdelta}{\mathbb{F}^r(\delta_1, \delta_2)}
\newcommand{\Fmn}{\mathbb{F}^r(m,n)}

\usetikzlibrary{decorations.pathreplacing}

\usepackage{amsmath}
\mathchardef\mhyphen="2D

\DeclareMathOperator{\node}{node}

\let\originalleft\left
\let\originalright\right
\def\left#1{\mathopen{}\originalleft#1}
\def\right#1{\originalright#1\mathclose{}}

\renewcommand{\ge}{\geqslant}

\renewcommand{\ge}{\geqslant}
\renewcommand{\geq}{\geqslant}
\renewcommand{\le}{\leqslant}
\renewcommand{\leq}{\leqslant}

  \newcommand{\pgen}{p}
    \newcommand{\egen}{e}
      \newcommand{\sgen}{s}

\def\ignore#1{\relax}

\def\ignore#1{\relax}

\newcommand\encircle[1]{%
  \tikz[baseline=(X.base)] 
    \node (X) [draw, shape=circle, inner sep=0] {\strut #1};}

\renewcommand{\theta}{\alpha}

\usepackage{scalefnt}
\begin{document} 
 \title[The partition algebra and the plethysm coefficients, I] {The partition algebra and the plethysm coefficients  I: Stability and Foulkes' conjecture }
 \author{Chris Bowman}
       \address{Department of Mathematics, 
University of York, Heslington, York,  UK}
\email{Chris.Bowman-Scargill@york.ac.uk}
  \author{Rowena Paget}

\address{School of Mathematics, Statistics and Actuarial Science University of Kent, 
CT2 7NF, UK}
 \email{R.E.Paget@kent.ac.uk}
 \maketitle
  
 \renewcommand{\Bbbk}{\mathbb{C}}
 
 \maketitle
 
 \vspace{-0.2cm}
 \begin{abstract}
 We propose a new approach to study plethysm coefficients by using the Schur-Weyl duality between the symmetric group and the partition algebra.  
This provides an explanation of the stability properties of plethysm and Kronecker coefficients in a simple and uniform fashion for the first time.  
We prove the strengthened Foulkes'  conjecture  for stable plethysm coefficients in an elementary fashion.  
 \end{abstract}
  \vspace{-0.2cm}
  \section*{Introduction}
Understanding the \emph{plethysm coefficients} is a fundamental problem in the representation theories of symmetric and general linear groups 
and was identified by Richard Stanley as one of the most important open problems in algebraic combinatorics \cite{MR1754784}.   
Perhaps the oldest  and most famous   question  concerning plethysm coefficients is a
 conjecture of  Foulkes
 from 1950 \cite{MR0037276}.  
To state Foulkes' Conjecture, we first require some notation.  
Let $m,n \in \mathbb{N}$ and $\alpha$ be a partition of $mn$ and let $\Sym _m\wr \Sym _n$ denote the wreath product subgroup of $\Sym _{mn}$.  The plethysm coefficient
$p((n), (m), \alpha)$ is the multiplicity of the irreducible $\CC \Sym _{mn}$-module ${\sf S}(\alpha)$ as a composition factor of the Foulkes module $\ind_{\Sym _m\wr \Sym _n}^{\Sym _{mn}} \CC$.
 Equivalently,  these plethysm coefficients  record the decomposition of the  $\GL_{mn}(\CC)$-module ${\rm Sym}^n({\rm Sym}^m \! (\CC^{mn}))$ into irreducible summands and also the decomposition of
the plethysm of symmetric functions $s_{(n)} \circ s_{(m)}$ as an integral linear combination of Schur functions.
 Foulkes' Conjecture states, for all $m\leq n$ and for all $\alpha\vdash mn$, that
\begin{equation}\tag{1}\label{1}
p(  (m),(n) ,\alpha) \leq p((n), (m), \alpha).
\end{equation}
   A stronger conjecture made in \cite{MR2067621}  states that 
\begin{equation}\tag{2}\label{2}
p((q), (p), \alpha) \leq 
p((n), (m), \alpha)
\end{equation}
for all $m\leq n,p,q$ with $mn=pq$.  
Plethysm is defined for arbitrary  partitions of $m$ and $n$, but for the purposes of this paper our interest lies in the `Foulkes case' where
both partitions have precisely one row.  
In this article, we study families of these plethysm coefficients.   For an arbitrary partition  $\lambda=(\lambda_1,\lambda_2,\dots, \lambda_\ell)$ ,  set $\lambda_{[mn]}=(mn-|\lambda|,\lambda_1,\lambda_2,\dots, \lambda_\ell)$. We look at the coefficient 
$p((n), (m), \lambda_{[mn] })$
as $m$ and $n$ vary.  We 
 ask, for a fixed partition $\lambda$,  whether
$$ p((m), (n),  \lambda_{[mn]} ) \leq p((n), (m),  \lambda_{[mn] })$$ for all $m<n$.
 Our first   theorem verifies that
 this is indeed the case  for all except possibly a finite list of values for  $m, n \in \mathbb{N}$.
 Discarding this finite list of values, both Foulkes' Conjecture  and the strengthened Foulkes' Conjecture hold for the partition $\lambda_{[mn]}$.
 Moreover, we can even drop the assumption in \cref{2} that $mn=pq$, relating plethysm coefficients for $\lambda_{[mn]}$ and $\lambda_{[pq]}$ outside of these values.

   \begin{thm1} 
   \label{thm1}
 Let $\lambda$ be an arbitrary partition.  
 For any  $m,n,p,q \geq |\lambda|$, 
 $$  p(   ( q), (p), \lambda_{[pq]})=   p(   ( n), (m), \lambda_{[mn]}). $$   
 In particular, taking $p=n$ and $q=m$,   Foulkes'   Conjecture  holds for $\lambda_{[mn]}$ for all but finitely many values of  $m,n\in \mathbb{N}$, as does the strengthened Foulkes' Conjecture.  
     \end{thm1}
     
     The proof of this result constructs a partition algebra isomorphism which ``does not see" any difference between $m$ and $n$ providing they are both sufficiently large.  This seems to provide the first conceptual explanation for why Foulkes' conjecture ``should" be true.

 One of  the key ideas in our approach is to consider the stable limit of a certain sequence of plethysm coefficients.  
  Brion \cite{brion} and Carr\'e--Thibon \cite{MR1190119} proved that the following sequences of plethysm coefficients 
 $$
\{ p(   ( n), (m), \lambda_{[mn]})  \}_{n \in \mathbb{N}}
\qquad 
\{ p(   ( n), (m), \lambda_{[mn]})\}_{m \in \mathbb{N}}
 $$
 have stable limits for $n$ (respectively $m$) sufficiently large with respect to $m$ (respectively $n$).   In fact, Brion's proof of the stability of the former sequence  
settled  a second conjecture from  Foulkes' 1950  paper. 
In this paper we consider the stable limit of the  double-sequence 
 \begin{equation}\label{dagger}\tag{$\dagger$}
\overline{p}_{\infty,\lambda}=\lim_{m,n\to \infty}  \{p(   ( n), (m), \lambda_{[mn]})		\}.
\end{equation}
These stable values are  achieved whenever $m,n\geq |\lambda|$.  
 We study these stable plethysm coefficients through the Schur--Weyl duality between the   symmetric group, $  \Sym _{mn}$, and the partition algebra, $P_r(mn)$, via their actions on the tensor space $(\mathbb{C}^{mn})^{\otimes r}$.   This duality results in a   functor $$\mathcal{F}_r : 
\Sym _{mn}\mhyphen{\rm mod} \to{\rm mod}  \mhyphen P_r(mn).$$   
The key observation is that the module $\mathcal{F}_r(\ind_{\Sym _m\wr \Sym _n}^{\Sym _{mn}}(\Bbbk))$ 
has an elegant diagrammatic description for $m,n\geq r$.  
 By considering the module $\mathcal{F}_r(\ind_{\Sym _m\wr \Sym _n}^{\Sym _{mn}}(\Bbbk))$ 
  for each  $r\geq 0$ in turn, 
   we are able to focus solely on the  `$r$th  layer' of plethysm constituents 
 $p(   ( n), (m), \lambda_{[mn]})$ for which  $|\lambda|= r$.    We hence deduce that  Foulkes' conjecture holds for $\lambda_{[mn]}$ whenever  $m,n\geq |\lambda|$
and obtain a simple proof of the stability (\ref{dagger}).

   Our second main result (which subsumes Theorem A) calculates the value of any  {stable plethysm coefficient} in terms of 
  plethysm coefficients labelled by much smaller partitions  and use of the Littlewood--Richardson rule.  
A  geometric proof of this result in the language of jet schemes is given in \cite{MR1651092}.  
 Our proof of this result is given by explicitly decomposing the $P_r(mn)$-module  $\mathcal{F}_r(\ind_{\Sym _m\wr \Sym _n}^{\Sym _{mn}}(\Bbbk))$ using its diagrammatic incarnation.

    \begin{thm2}\label{thm2}
   Let  $\lambda$ be a partition of $r\in \mathbb{N}$.  For all $m,n \geq r$ we have that
   $$
   \overline{p}_{\infty,\lambda } = p(   ( n), (m), \lambda_{[mn]})= \sum_{\boldsymbol
   \mu \in \mathscr{P}_1(r)
   }p_{\boldsymbol\mu}( \lambda ),
   $$
   where $ \mathscr{P}_1(r)$ is the set of all partitions of $r$ 
   whose parts are all strictly greater than~1.   The coefficients $p_{\boldsymbol\mu}( \lambda )$
are the generalised plethysm coefficients defined in \cref{sec1}.
  \end{thm2}

Our approach brings forward a general tool to study stable and non-stable plethysm coefficients and provides a natural framework for the study of the outstanding problems in the area. In particular, one should notice that our proofs are surprisingly elementary and treat the stabilities of Kronecker  and plethysm coefficients uniformly alongside one another for the first time ---  as   the parameters  increase, the action of the partition algebra becomes faithful and semisimple exactly as in the case of the Kronecker coefficients  \cite{BDO15}.  
We consider Theorem B to be the natural analogous statement   to that for Kronecker coefficients in   \cite[Corollary 4.5]{BDO15}.
 Particular highlights of our approach include easy algebraic proofs of  Foulkes'  and Weintraub's conjectures  for stable plethysm coefficients.   The partition algebra approach has proven to be very powerful for understanding Kronecker coefficients: in \cite{BDE}  an algorithm is given for  calculating stable Kronecker coefficients in terms of oscillating tableaux. (The partition algebra is essential in the proof and allows one to define a lattice permutation condition on oscillating tableaux.)  We hope that the partition algebra is similarly useful in understand the (stable) plethysm coefficients.

\subsection*{Ramifications}
In this paper, we recast the Foulkes' plethysm coefficients in the setting of the partition algebra.  In the sequel to this paper we will generalise this to arbitrary plethysm coefficients.  The key to this construction will be  the {\em ramified}   partition algebra of \cite{INSERT} which we do not discuss here.  However, the reader familiar with these constructions is invited to observe that  our {\em stable Foulkes module} for $P_r(mn)$ is equal to the restriction to  $P_r(mn)$ of the cell module of the ramified partition algebra denoted by $\Delta_r(\varnothing^\varnothing)$.

\subsection*{An example}
 We conclude this introduction with an example, illustrating how to calculate the multiplicities $   \overline{p}_{\infty,\lambda}=p(   ( n), (m), \lambda_{[mn]})$ for $\lambda$ a partition of ~4 and $m,n\geq 4$.
We pass the question from $\Sym _{mn}$ to the partition algebra $P_4(m n)$  and take  the natural quotient   $P_4(m n) \to \CC\Sym _4$, 
which dramatically reduces the rank of the problem.  
We hence obtain an $\CC\Sym _4$-module  with the following basis
  $$\scalefont{0.8}
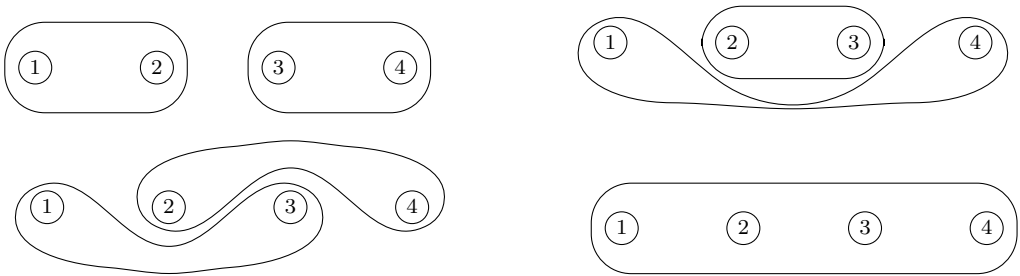
 \begin{tikzpicture}[scale=0.8]
    \foreach \x in {0,2,4,6}
        \fill[white](\x,3.5) circle (10pt);   
                   \draw (2,3.5) node {$\encircle{2}$}; 
                           \draw (0,3.5) node {$\encircle{1}$}; 
                                \draw (4,3.5) node {$\encircle{3}$}; 
                           \draw (6,3.5) node {$\encircle{4}$}; 
     \draw[rounded corners=15pt]
  (-0.5,2.75) rectangle ++(3,1.5);
     \draw[rounded corners=15pt]
  (-0.5+4,2.75) rectangle ++(3,1.5); 
      \end{tikzpicture}\qquad\qquad\qquad
  \begin{tikzpicture}[scale=0.8]
    \foreach \x in {0,2,4,6}
        \fill[white](\x,0) circle (10pt);   
                   \draw (2,0) node {$\encircle{2}$}; 
                           \draw (0,0) node {$\encircle{1}$}; 
                                \draw (4,0) node {$\encircle{3}$}; 
                           \draw (6,0) node {$\encircle{4}$}; 
         \path[draw,use Hobby shortcut,closed=true]
 (-0.5,0)..(-.25,.3)..(1.8,-0.65)..(4.2,-0.65)..(6.25,.3)..(6.5,0)..(5,-1)..(3,-1.1)..(1,-1)..(-0.5,0);
   \draw[rounded corners=15pt]
  (2-.5,0-.6) rectangle ++(3,1.2);
   \end{tikzpicture}
      $$
$$\scalefont{0.8}
 \begin{tikzpicture}[scale=0.8]
    \foreach \x in {0,2,4,6}
        \fill[white](\x,0) circle (10pt);   
                   \draw (2,0) node {$\encircle{2}$}; 
                           \draw (0,0) node {$\encircle{1}$}; 
                                \draw (4,0) node {$\encircle{3}$}; 
                           \draw (6,0) node {$\encircle{4}$}; 
         \path[draw,use Hobby shortcut,closed=true]
 (-0.5,0)..(-.25,.3)..(2,-0.65)..(4.25,.3)..(4.5,0)..(3,-1)..(2,-1.1)..(1,-1)..(-0.5,0);

  \path[draw,use Hobby shortcut,closed=true]
 (-0.5+2,0)..(-.25+2,-.3)..(2+2,0.65)..(4.25+2,-.3)..(4.5+2,-0)..(3+2,1)..(2+2,1.1)..(1+2,1)..(-0.5+2,0);
    \end{tikzpicture}
\qquad\qquad\qquad  
   \begin{tikzpicture}[scale=0.8]
    \foreach \x in {0,2,4,6}
        \fill[white](\x,3.5) circle (10pt);   
                   \draw (2,3.5) node {$\encircle{2}$}; 
                           \draw (0,3.5) node {$\encircle{1}$}; 
                                \draw (4,3.5) node {$\encircle{3}$}; 
                           \draw (6,3.5) node {$\encircle{4}$}; 
     \draw[rounded corners=15pt]
  (-0.5,2.75) rectangle ++(7,1.5);
       \end{tikzpicture}
       $$
with the action given by permuting the diagrams in the obvious fashion. It is easy to see that the first three diagrams 
span a cyclic module which decomposes as the sum of Specht modules 
 ${\sf S}(4)\oplus   {\sf S}(2,2)$. 
The fourth diagram provides an indecomposable module isomorphic to 
${\sf S}(4)$.  
Hence, for all $m,n\geq 4$ we deduce that  
  $$p  (( n),(m), (mn-4,4))=2 \qquad   p  (( n),(m), (mn-4,2^2))=1 \qquad  p(   ( n), (m), \lambda_{[mn]})=0$$ for $\lambda =(3,1),(2,1^2),(1^4)$.  

  \subsection*{The structure of the paper}
 We present sufficient background to make our exposition accessible to a reader familiar with representation theory: \cref{sec1} recalls the definition of the (generalised) plethysm coefficients, their stabilities, 
and the  statement of Foulkes' conjecture and   the
strengthened Foulkes' conjecture; \cref{sec2} recalls the definition of the partition algebra and the basic facts concerning its representation theory;
\cref{sec3} recalls the Schur--Weyl duality between the symmetric groups and partition algebras which underlies the main results of this paper.  
We apply Schur--Weyl duality   to the Foulkes module to obtain a module for the partition algebra.
Our fundamental combinatorial object, the fixed depth Foulkes poset, is introduced in   \cref{sec:poset}.  It is used to construct a diagrammatic module for the partition algebra in \cref{sec:diagrammatic}, which we relate to Foulkes modules in \cref{relateS}.
 In \cref{sec5} the construction of a filtration of our diagrammatic module allows us to prove Theorem A, and  its decomposition into its irreducible components  proves Theorem B.  


\bigskip

 This paper has been on the arXiv for quite some time, while we were preparing 
 the sequel.  Since then, certain plethysm coefficients have been studied by
 Orellana--Saliola--Schilling--Zabrocki in the context of the {\em party algebra}, a subalgebra of the partition algebra (see
 \cite{party}).  
We do not believe there is any overlap in our results, but 
the ideas do have a similar diagrammatic flavour.

\begin{Acknowledgements*}
The authors would like to thank Mark Wildon for providing us with a wealth of small rank examples which he calculated using his plethysm coefficients software.  We would also like to thank Rosa Orellana and Mike Zabrocki for interesting and informative conversations and the Oberwolfach workshop 
``Character Theory and Categorification" for providing an excellent environment for collaboration.  
The first author is grateful to financial support from EPSRC grant EP/V00090X/1.
\end{Acknowledgements*}  

  \section{Foulkes' conjecture and plethysm coefficients}\label{sec1}
  
  We let $\Sym _{n}$ denote the symmetric group on $n$ letters.   
     The combinatorics underlying the representation theory of the partition algebras and   symmetric 
     groups  is based on (integer) compositions and   partitions.  
A \emph{composition}    $\lambda $ of $n$, denoted $\lambda \vDash n$, is defined to be a  sequence   of non-negative integers which  sum to $n$.  
  If the sequence is weakly decreasing, we write  $\lambda \vdash n$
  and refer to $\lambda$ as a \emph{partition} of $n$.  We let 
  $\mathscr{P}(n)$ denote the set of all partitions of $n$ and we let   $\mathscr{P}_1(n)$ denote the subset of those partitions whose (non-zero) parts are all strictly greater than 1.  (We shall not write down the zero parts.)
We say that a partition $\lambda=(\lambda_1,\lambda_2,\lambda_3,\dots,\lambda_\ell)$ has {\em depth} equal to $\lambda_2+\lambda_3+\dots+\lambda_\ell$.  We write $\emptyset$ for the unique partition of zero.
%
%

Associated to each   partition $\lambda$ of $n$, we have a        simple    right  $\mathbb{C}\Sym _{n}$-module, often referred to as a {\em Specht} module. 
An explicit construction of these modules is given in \cite[\S4]{james}, where it is shown that 
the Specht modules provide a complete set of irreducible $\CC\Sym _{n}$-modules  indexed by the partitions of $n$.
  
Now let $m,n\in \mathbb{N}$ and consider the symmetric group   $\Sym _{mn}$.
Given  $\lambda$ and $\boldsymbol\mu=(m_1^{n_1},m_2^{n_2},\dots, m_\ell^{n_\ell})  $,   partitions of $ mn$, 
we  suppose that $m_1>m_2> \dots >m_\ell$.  
Associated to $\boldsymbol\mu$ we have a subgroup 
%
%
\begin{equation}\label{genwreathprodsubgroup}
\prod_i\Sym _{m_i}\wr \Sym _{n_i}= \Sym _{m_1}\wr \Sym _{n_1}\times 
\dots \times  
   \Sym _{m_\ell}\wr \Sym _{n_\ell}  
   \leq \Sym _{mn}
   \end{equation}
   and we define the generalised plethysm coefficient $ p_{\boldsymbol\mu}(\lambda)$  
to be the multiplicity of ${\sf S}(\lambda)$ as a composition factor of the permutation module $\ind^{\Sym _{mn}}_{\prod_i \Sym _{m_i}\wr \Sym _{n_i} }(\Bbbk)$,
\begin{equation}\label{genplethy}
   p_{\boldsymbol\mu}(\lambda)=
\left[  \ind^{\Sym _{mn}}_{\prod_i\Sym _{m_i}\wr \Sym _{n_i}}(\Bbbk) \, : \,  {\sf S}(\lambda))\right]_{\CC \Sym _{mn}}.
\end{equation}
   In the special case of a rectangular partition, $\boldsymbol\mu=(m^n)$, the subgroup above specialises to be 
 $\Sym _{m}\wr \Sym _{n}$ and  we hence obtain the classical plethysm coefficients  $p_ {(m^n)}(\lambda)= p((n),(m),\lambda)$ defined explicitly in the introduction.  
  
     Given a fixed integer $ mn\in \mathbb{N}$, the plethysm coefficients (associated to $\boldsymbol\mu=(m^n)$)  are the most 
     difficult examples of the coefficients in \cref{genplethy}.  Indeed, all other examples of coefficients 
       in \cref{genplethy} can be obtained from  an understanding of the  smaller rank plethysm coefficients and applications of    the Littlewood--Richardson rule (for the statement of which, see \cite[\S16]{james}).  To see this, simply note that we are inducing from a product (hence the Littlewood--Richardson coefficients) of wreath product subgroups (hence the plethysm coefficients).

Recall from the introduction that if  $\lambda$ is a partition then $\lambda_{[mn]}$ denotes the partition of $mn$ whose Young diagram is obtained by appending an additional row above those of $\lambda$.
Note that all partitions of $mn$ can be written in this form.
 Brion    \cite{brion}  showed   that if we allow both the value of  $m$   and hence the length of the first row of $\lambda_{[mn]}$ to increase, then we obtain a limiting behaviour as follows. 
  For $m$ sufficiently large with respect to $n$,
   Brion proved  that
$$
p(   ( n+k), (m), \lambda_{[m(n+k)]}) 
=p(   ( n), (m), \lambda_{[mn]}) $$ 
for all $k\geq 1$.  This 
stability was conjectured by Foulkes in \cite{MR0037276}.      
 In the other direction, Carr\'e--Thibon \cite {MR1190119} showed for $n$ sufficiently large with respect to $m$,
   we have that
$$
p(   ( n ), (m+k), \lambda_{[ (m+k)n]})  =p(   ( n), (m), \lambda_{[mn]}) $$
for all $k\geq 1$.  
Notice that in each case we require that $m$ (respectively $n$) is sufficiently large with respect to 
$n$ (respectively $m$).  
Therefore one cannot, a priori, consider the limit as  $n$ and $m$  both tend to infinity.  
In this paper we shall consider the stability of the double-sequence 
$$
\overline{p}_{\infty}(\lambda)=\lim_{m,n\to \infty}  \{p(   ( n), (m), \lambda_{[mn]})\}. 
$$
and  we shall see that   $n$ and $m$ 
are only required to be greater than $|\lambda|$ (and can be chosen freely with respect to each other) for this stability to occur.

We now recall Foulkes' original conjecture from \cite{MR0037276}.  
This conjecture has been settled for $m\ll n$ by Brion   \cite{brion}.
%
All other results in the area concern small values of $m$;  namely, 
$m=2$  \cite{thrall}, $m=3$ \cite{dent}, $m=4$ \cite{mckay}, $m=5$ \cite{MR3574536}.

\begin{conj} {\bf (Foulkes' Conjecture)}
Let  $m, n \in \mathbb{N}$ with $m\leq n$.
Then   $$p(   ( m), (n), \lambda_{[mn]})\leq
p(      (n),( m), \lambda_{[mn]})$$ for all partitions $\lambda$.  
\end{conj}

A strengthened version of this conjecture was proposed in \cite{MR2067621}.  It is currently only known to hold  for $m=2$ and $n=3$.  
\begin{conj}  {\bf (Strengthened Foulkes' Conjecture)}
Let  $ mn=pq \in \mathbb{N}$ with $m\leq n$ and suppose that $p,q\geq m$.  
Then $$p(    (m), ( n), \lambda_{[mn]})\leq 
p(    (q), ( p), \lambda_{[pq]})
$$ for all partitions $\lambda$.  
\end{conj}

  \section{The partition algebra}\label{sec2}
  The partition algebra was originally defined by Martin in \cite{marbook}.  All the results in this section are due to Martin and his collaborators: see \cite{mar1} and references therein.  
%
   
 \subsection{Definitions}
 For $r\in\mathbb{N}$, we consider the set $\{1,2,\ldots, r, \bar{1},\bar{2}, \ldots, \bar{r}\}$ with the total ordering 
 $$
 1<2<\dots <r <\overline{1}<\overline{2}<\dots <\overline{r}.
 $$  
We consider set-partitions of $\{1,2,\ldots, r, \bar{1},\bar{2}, \ldots, \bar{r}\}$.
A subset occurring in a set-partition is called a \emph{block}.
For example,
$$x=\{\{1, 2, 4, \bar{2}, \bar{5}\}, \{3\}, \{5, 6, 7, \bar{3}, \bar{4}, \bar{6}, \bar{7}\}, \{8, \bar{8}\}, \{\bar{1}\}\},$$
is a set-partition of  $\{1,2,\ldots, 8, \bar{1},\bar{2}, \ldots, \bar{8}\}$ with five  blocks.
By convention, we order the  blocks in 
$x = \{X_1, \dots, X_l\}$
 by \emph{increasing minima}, so that
\[
1 =   \min X_1 < \min X_2 < \cdots < \min \Lambda_{l-1} <    \min X_l   \leq  \overline{r}. 
\]
A  set-partition, $x$,
 can be represented  
 by an $(r,r)$-\emph{partition diagram}, 
 consisting of  $r$ distinguished    northern and southern points, which we call vertices.  We number the northern vertices from left to right by $1,2,\ldots, r$ and the southern vertices similarly by $\bar{1},\bar{2},\ldots, \bar{r}$ and connect two vertices by an edge if they belong to the same block and are adjacent in the total ordering 
  given by restriction of the above ordering to the given block. 
The second condition is imposed to pick a unique representative from the equivalence class of all diagrams having the same connected components. 
We shall move between set-partitions and their diagrams without further comment.
For example, the diagram of the set-partition $x$ given above is as follows:
 $$\scalefont{0.8}
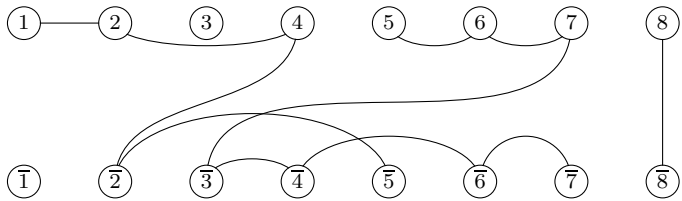
 \begin{tikzpicture}[scale=0.6]
    \draw  (0,3.5) arc (180:360:1 and 0);
      \draw  (2,3.5) arc (180:360:2 and .5);
       \draw  (8,3.5) arc (180:360:1 and .5);
              \draw  (10,3.5) arc (180:360:1 and .5);
        \draw  (2,0) arc (180:360:3 and -1.5);
                \draw  (4,0) arc (180:360:1 and -0.5);
        \draw  (6,0) arc (180:360:2 and -1);      
                \draw  (10,0) arc (180:360:1 and -1);  
        \draw  (14,0) to [out=90,in=-90] (14,3.5);   
                \draw  (2,0)to [out=90,in=-90] (6,3.5);   
                                \draw  (4,0) to [out=90,in=-90] (12,3.5);   
  \foreach \x in {0,2,...,16}
        \fill[white](\x,3.5) circle (10pt);   
      \draw (4,3.5) node {$ \encircle{3}  $};   
                  \draw (2,3.5) node {$\encircle{2}$}; 
                           \draw (0,3.5) node {$\encircle{1}$}; 
         \draw (6,3.5) node {$\encircle{4}$};   
                  \draw (8,3.5) node {$\encircle{5}$};   
                           \draw (10,3.5) node {$\encircle{6}$};    
         \draw (12,3.5) node {$\encircle{7}$};       \draw (14,3.5) node {$\encircle{8}$}; 
  \foreach \x in {0,2,...,16}
        \fill[white](\x,0) circle (10pt);   
      \draw (4,0) node {\encircle{$\overline{3}$}}; 
            \draw (2,0) node {\encircle{$\overline{2}$}}; 
       \draw (0,0) node {\encircle{$\overline{1}$}};      
             \draw (6,0) node {\encircle{$\overline{4}$}}; 
            \draw (8,0) node {\encircle{$\overline{5}$}}; 
       \draw (10,0) node {\encircle{$\overline{6}$}};    
             \draw (12,0) node {\encircle{$\overline{7}$}}; 
            \draw (14,0) node {\encircle{$\overline{8}$}};    
          
     \end{tikzpicture}.
$$

We can generalise this definition to 
$(m,r)$-partition diagrams as diagrams representing set-partitions of $\{1, \ldots , m, \bar{1}, \ldots, \bar{r}\}$ in the obvious way.  

We now consider a parameter  $\delta \in \CC$.  We define the product $x  y$ of two $(r,r)$-partition diagrams $x$
and $y$ using the concatenation of $x$ above $y$, where we identify
the southern vertices of $x$ with the northern vertices of $y$.   
If there are $t$ connected components consisting only of  middle vertices, then the product is set equal to $ \delta ^t$ times the 
$(r,r)$-partition diagram equivalent to the 
diagram with the middle components removed. 
We let $P_r(\delta)$ denote the complex vector space with basis given by  all set-partitions of  $\{1,2,\ldots, r, \bar{1},\bar{2}, \ldots, \bar{r}\} $  
and with multiplication given by  linearly extending  the multiplication of diagrams.   
An example of the multiplication is given in \cref{hereitissss}. 
\begin{figure}[ht!]
 $$\scalefont{0.8}
\begin{minipage}{5.8cm} \begin{tikzpicture}[scale=0.6]

\draw[thick,densely dotted] (0,-1.1)--++(90:1.1);
\draw[thick,densely dotted] (2,-1.1)--++(90:1.1);
\draw[thick,densely dotted] (4,-1.1)--++(90:1.1);
\draw[thick,densely dotted] (6,-1.1)--++(90:1.1);
\draw[thick,densely dotted] (8,-1.1)--++(90:1.1);

    \draw  (0,3.5) arc (180:360:1 and 0);
      \draw  (2,3.5) arc (180:360:2 and .5);
        \draw  (2,0) arc (180:360:3 and -1.5);
                \draw  (4,0) arc (180:360:1 and -0.5);
                 \draw  (2,0)to [out=90,in=-90] (6,3.5);   

 \foreach \x in {0,2,...,16}
        \fill[white](\x,3.5) circle (10pt);   
      \draw (4,3.5) node {$ \encircle{3}  $};   
                  \draw (2,3.5) node {$\encircle{2}$}; 
                           \draw (0,3.5) node {$\encircle{1}$}; 
         \draw (6,3.5) node {$\encircle{4}$};   
                  \draw (8,3.5) node {$\encircle{5}$};   
   \foreach \x in {0,2,...,16}
        \fill[white](\x,0) circle (10pt);   
      \draw (4,0) node {\encircle{$\overline{3}$}}; 
            \draw (2,0) node {\encircle{$\overline{2}$}}; 
       \draw (0,0) node {\encircle{$\overline{1}$}};      
             \draw (6,0) node {\encircle{$\overline{4}$}}; 
            \draw (8,0) node {\encircle{$\overline{5}$}}; 

       \draw  (2,-1.1)--++(-90:3.75);
    \draw  (0,-1.1) arc (180:360:1 and 0);
        \draw  (2,-3.75+0.25-1.1) arc (180:360:3 and -1.5);
                \draw  (4,-3.75+0.25-1.1) arc (180:360:1 and -0.5);

 \foreach \x in {0,2,...,16}
        \fill[white](\x,-1.1) circle (10pt);   
      \draw (4,-1.1) node {$ \encircle{3}  $};   
                  \draw (2,-1.1) node {$\encircle{2}$}; 
                           \draw (0,-1.1) node {$\encircle{1}$}; 
         \draw (6,-1.1) node {$\encircle{4}$};   
                  \draw (8,-1.1) node {$\encircle{5}$};   
   \foreach \x in {0,2,...,16}
        \fill[white](\x,-3.75+0.25-1.1) circle (10pt);   
      \draw (4,-3.75+0.25-1.1) node {\encircle{$\overline{3}$}}; 
            \draw (2,-3.75+0.25-1.1) node {\encircle{$\overline{2}$}}; 
       \draw (0,-3.75+0.25-1.1) node {\encircle{$\overline{1}$}};      
             \draw (6,-3.75+0.25-1.1) node {\encircle{$\overline{4}$}}; 
            \draw (8,-3.75+0.25-1.1) node {\encircle{$\overline{5}$}};

     \end{tikzpicture}\end{minipage}
     =
\; \; \delta \times \;\;\;
  \begin{minipage}{5.8cm} \begin{tikzpicture}[scale=0.6]
    \draw  (0,3.5) arc (180:360:1 and 0);
      \draw  (2,3.5) arc (180:360:2 and .5);
        \draw  (2,0) arc (180:360:3 and -1.5);
                \draw  (4,0) arc (180:360:1 and -0.5);
                 \draw  (2,0)to [out=90,in=-90] (6,3.5);   

 \foreach \x in {0,2,...,16}
        \fill[white](\x,3.5) circle (10pt);   
      \draw (4,3.5) node {$ \encircle{3}  $};   
                  \draw (2,3.5) node {$\encircle{2}$}; 
                           \draw (0,3.5) node {$\encircle{1}$}; 
         \draw (6,3.5) node {$\encircle{4}$};   
                  \draw (8,3.5) node {$\encircle{5}$};   
   \foreach \x in {0,2,...,16}
        \fill[white](\x,0) circle (10pt);   
      \draw (4,0) node {\encircle{$\overline{3}$}}; 
            \draw (2,0) node {\encircle{$\overline{2}$}}; 
       \draw (0,0) node {\encircle{$\overline{1}$}};      
             \draw (6,0) node {\encircle{$\overline{4}$}}; 
            \draw (8,0) node {\encircle{$\overline{5}$}}; 
  
     \end{tikzpicture}\end{minipage}
$$
\caption{An example of the multiplication in $P_5(\delta)$. }
\label{hereitissss}
\end{figure}

 We set 
 $$  
\pgen _1=\{
\{1\}  ,
  \{2,\overline{2}\} , \dots  \{r,\overline{r}\}
  ,\{\overline1\}
\}, \quad
\pgen _{1,2}=\{
\{1,2,\overline{1},\overline2\} ,
  \{3,\overline{3}\} , \dots  \{r,\overline{r}\}
\}
$$ 
 and  we recall that the 
   Coxeter generators  for the symmetric group
 $\sgen _{i,i+1}$ for 
  $1\leq i<   r$  can be thought of  as the set-partitions 
  $$\sgen _{i,i+1}
=\{
\{1,\overline{1}\} 
, \dots,  \{i-1,\overline{i-1}\} , \{i,\overline{i+1}\}, 
 \{i+1,\overline{i}\},
 \{i+2,\overline{i+2}\},\dots , \{r,\overline{r}\}
\}.$$ 
We set $\pgen _{k}=
\sgen _{k,k-1}\dots \sgen _{1,2} \pgen _{1}\sgen _{1,2} \dots \sgen _{k,k-1}$.  Some of these diagrams are pictured in \cref{generators}.

\begin{prop}[{\cite[Proposition 1]{mar1}}]
The algebra $P_r(\delta)$ is generated by   the set-partitions 
$\pgen _1, \pgen _{1,2}$  and $\sgen _{i,i+1}$ for $1\leq i <r$.  
   \end{prop}

\begin{figure}[h!]
  $$ 
 \scalefont{0.8}
\begin{minipage}{80mm}\begin{tikzpicture}[scale=0.6]
   \draw  (2,3.5) --(2,0);  \draw  (4,3.5) --(4,0);
    \draw  (8,3.5) --(8,0); 
  \foreach \x in {0,2,4,8}
        \fill[white](\x,3.5) circle (10pt);   
      \draw (4,3.5) node {$ \encircle{3}  $};   
                  \draw (2,3.5) node {$\encircle{2}$}; 
                           \draw (0,3.5) node {$\encircle{1}$}; 
     \draw (6,0) node {$\dots $};          \draw (6,3.5) node {$\dots $};                             \draw (8,3.5) node {$\encircle{$r$}$};    
   \foreach \x in {0,2,4,8}
        \fill[white](\x,0) circle (10pt);   
      \draw (4,0) node {  \encircle{$\overline{3}$}};   
                 \draw (2,0) node { \encircle{$\overline{2}$}}; 
                           \draw (0,0) node {\encircle{$\overline{1}$}}; 
                            \draw (8,0) node {\encircle{$\overline{r}$}};    
      \end{tikzpicture}
      \end{minipage}
    \begin{minipage}{60mm}  \begin{tikzpicture}[scale=0.6]
   \draw  (2,3.5) --(2,0);  \draw  (4,3.5) --(4,0);
    \draw  (8,3.5) --(8,0); 
        \draw  (0,3.5) --(-2,0); 
        \draw  (-2,3.5) arc (180:360:1 and 0.5);
                \draw  (-2,0) arc (180:360:1 and -0.5);
  \foreach \x in {-2,0,2,4,8}
        \fill[white](\x,3.5) circle (10pt);   
      \draw (4,3.5) node {$ \encircle{4}  $};   
                  \draw (2,3.5) node {$\encircle{3}$}; 
                           \draw (-2,3.5) node {$\encircle{1}$};                            \draw (0,3.5) node {$\encircle{2}$}; 
     \draw (6,0) node {$\dots $};          \draw (6,3.5) node {$\dots $};                             \draw (8,3.5) node {$\encircle{$r$}$};    
   \foreach \x in {-2,0,2,4,8}
        \fill[white](\x,0) circle (10pt);   
      \draw (4,0) node {\encircle{$\overline{4}  $}};   
                 \draw (2,0) node {\encircle{$\overline{3}$}}; 
                           \draw (-2,0) node {\encircle{$\overline{1}$}}; 
                           \draw (0,0) node {\encircle{$\overline{2}$}}; 
                            \draw (8,0) node {\encircle{$\overline{r}$}};    
      \end{tikzpicture}      \end{minipage}
 $$
\caption{The non-Coxeter generators, $\pgen _1$ and  $\pgen _{1,2}$,  of $P_r(\delta)$, respectively. }
\label{generators}
\end{figure}

  \subsection{Filtration by propagating blocks and standard modules}\label{2.2} 
A block of a set-partition of $\{1,2,\ldots, r, \bar{1},\bar{2}, \ldots, \bar{r}\}$ is called \emph{propagating} if the block contains both northern and southern vertices. In the example from the previous subsection, $x$ has three propagating blocks.
   Note that the multiplication in $P_r(\delta) $ cannot increase the number of propagating blocks.  More precisely, if $x$,  respectively $y$, is a partition diagram with $p_x$, respectively $p_y$, propagating blocks then $x y$ is equal to $ \delta ^t z$ for some $t\geq0$ and some partition diagram $z$ with $p_z$ propagating blocks, where $p_z\leq \min\{p_x, p_y\}$.  This gives a filtration of the algebra $P_r(\delta) $ by the number of propagating blocks.   Suppose now that $\delta \ne 0$. Then this filtration can be realised using the idempotents $\egen _k= { \delta ^{-k}}\pgen _1\pgen _2\ldots \pgen _k$ ($1\leq k\leq r$).  
  We have  
$$\{0\} \subset
 P_r(\delta)  \egen _r P_r(\delta)  \subset P_r(\delta)  \egen _{r-1}P_r(\delta)  \subset \ldots \subset  P_r(\delta) \egen _1P_r(\delta)  \subset P_r(\delta) .				$$
It is easy to see that
\begin{equation}\label{bob}
\egen _1P_r(\delta)  \egen _1 \cong  P_{r-1}(\delta ), 
 \end{equation}
and that this generalises to $P_{r-k}(\delta )\cong \egen _{k}P_r(\delta)  \egen _{k}$ for $1\leq   k  \leq r$. 
Moreover, $P_r(\delta) \egen _1P_r(\delta) $ is the span of all $(r,r)$-partition diagrams 
with at most $r-1$ propagating blocks and hence we have
\begin{align}\label{bob2}
 P_r(\delta)  / (P_r(\delta)  \egen _1P_r(\delta) ) \cong  \mathbb{C}    \Sym _r .
\end{align}
Using equation (\ref{bob2}), we see that any $\mathbb{C}\Sym _r$-module can be \emph{inflated} to a $P_r(\delta) $-module. 
 We also obtain from equations (\ref{bob}) and (\ref{bob2}), by induction, that the simple $P_r(\delta) $-modules are indexed by the set 
  $$\mathscr{P}(\leq r)=\bigcup_{0\leq i\leq r}\mathscr{P}(i).$$
%
 For any $\nu \in \mathscr{P}(\leq r)$ with $\nu \vdash r-k$, we define the \emph{standard} (right) $P_r(\delta) $-module, $\Delta_{r, \delta }(\nu)$, by
\begin{equation}\label{cell}
\Delta_{r, \delta }(\nu)= {\sf S}(\nu)\otimes_{P_{r-k}(\delta)} \egen _{k} P_r(\delta) ,
\end{equation}
where the (right) Specht module ${\sf S}(\nu)$  for $\CC \Sym _{r-k}$
 becomes a $P_{r-k}(\delta )$-module by inflation, and  we have identified $P_{r-k}(\delta )$ with $\egen _{k}P_r(\delta) \egen _{k}$ using the isomorphism given in equation (\ref{bob}) providing the left action on
 $\egen _{k} P_r(\delta) $.  The action of $P_r(\delta) $ on the standard module $\Delta_{r, \delta }(\nu)$ is given by right multiplication. As $P_{r-k}(\delta )$-modules, we have that  \begin{equation}\label{annihilate}
\Delta_{r, \delta }(\nu)\egen _k   \cong    \Delta_{r-k, \delta}(\nu) 
\end{equation}
  if $|\nu| \leq r-k$ and is zero otherwise.    

 It is known \cite{msannular} that $P_r(\delta) $ is semisimple if and only if 
\begin{equation}\label{ss}
\delta \not\in \{0,1,\ldots, 2r-2\}
\end{equation}
 and in this case the set $\{\Delta_{r, \delta }(\nu):\nu \in  \mathscr{P}(\leq r) \}$ forms a complete set of non-isomorphic simple $P_r(\delta) $-modules. 
In particular, if  $P_r(\delta) $  is semisimple then so is  $P_k(\delta) $ for $k \leq r$.
 More generally, provided $\delta  \ne 0$,  the standard module $ \Delta_{r, \delta }(\nu)$ has a simple head, which we denote by $L_{r, \delta }(\nu)$, and moreover  $$\{L_{r, \delta }(\nu):\nu \in \mathscr{P}(\leq r) \}$$ forms a complete set of non-isomorphic simple  $P_r(\delta) $-modules.
 
We now give an explicit description of the standard modules which follows directly from (\ref{cell}). We set $V(r-k,r)$ to be the span of all $(r-k,r)$-partition diagrams (that is, having $r-k$ northern and $r$ southern vertices) having precisely $(r-k)$ propagating blocks. This has the natural structure of a $(\Sym _{r-k}, P_r(\delta) )$-bimodule. It is easy to see that, as vector spaces, we have
$$\Delta_{r, \delta }(\nu) \cong {\sf S}(\nu) \otimes_{\Sym _{r-k}} V(r-k,r).$$ 
The action of $P_r(\delta) $ is given as follows. Let $v$ be a partition diagram in $V(r-k,r)$, $x\in {\sf S}(\nu)$ and $d$ be an $(r,r)$-partition diagram. Concatenate  $v$ above $d$ to get $(\delta)^t v'$ for some $(r-k,r)$-partition diagram $v'$ and some non-negative integer $t$. If $v'$ has fewer than $(r-k)$ propagating blocks then we set
$(x\otimes v)d=0$.
Otherwise we set $(x\otimes v)d= \delta ^t x\otimes v'$. 
 Note that if $\nu \vdash r$, then we have
\begin{equation}\label{annihilate2}
\Delta_{r, \delta }(\nu)\cong {\sf S}(\nu)\otimes_{\Sym _r}V(r,r) = {\sf S}(\nu),
\end{equation}
 viewed as a $P_r(\delta) $-module via equation  (\ref{bob2}).
A special case which will be important later is $\Delta_{r, \delta }(\emptyset)$ which can be viewed as the span of all set-partitions of $\{1,2,\ldots, r\}$ with the natural concatenation action.

\begin{rmk}
We refer the interested reader
 to \cite{BDO15} for explicit diagrammatic  calculations 
 using the modules constructed in \cref{annihilate2}.   
\end{rmk}

\section{The fixed depth Foulkes poset}
\label{sec:poset}
 
 For $r\in \mathbb{N}$, we consider the set $\{1,2,\ldots, r\}$ with the usual total ordering.
 Given  $\Lambda$   a set-partition of $\{1,2,\ldots, r\}$ and $a,b \in \{1,\dots, r\}$, we write $a \sim_\Lambda b$ if $a$ and $b$ belong to the same block of $\Lambda$, 
 and we denote the number of blocks of $\Lambda$  by $\ell(\Lambda)$.
For example, if  
$$\Lambda=\{\{1, 2, 4\}, \{3\}, \{5, 7, 8  \}, \{6,9\}\}$$
 then $1\sim_\Lambda 4$, $6\sim_\Lambda 9$ and $\ell(\Lambda)=4$.  
We represent this diagrammatically as follows. 
$$[\Lambda ]=\begin{minipage}{12cm}
\scalefont{0.8}
 \begin{tikzpicture}[scale=0.7]
    \draw  (0,3.5) arc (180:360:1 and 0);
      \draw  (2,3.5) arc (180:360:2 and .5);
   \draw  (12,3.5) arc (180:360:1 and 0);
      \draw  (8,3.5) arc (180:360:2 and .5);
        \draw  (10,3.5) arc (180:360:3 and -.5);

  \foreach \x in {0,2,...,16}
        \fill[white](\x,3.5) circle (10pt);   
      \draw (4,3.5) node {$ \encircle{3}  $};   
                  \draw (2,3.5) node {$\encircle{2}$}; 
                           \draw (0,3.5) node {$\encircle{1}$}; 
         \draw (6,3.5) node {$\encircle{4}$};   
                  \draw (8,3.5) node {$\encircle{5}$};   
                           \draw (10,3.5) node {$\encircle{6}$};    
         \draw (12,3.5) node {$\encircle{7}$};       \draw (14,3.5) node {$\encircle{8}$}; 
         \draw (16,3.5) node {$\encircle{9}$}; 
       \draw[rounded corners=15pt,white]
  (-0.75,2.75) rectangle ++(7.5,1.5);
   \draw[rounded corners=15pt,white]
  (7.25,2.75) rectangle ++(9.5,1.5);         \end{tikzpicture} \end{minipage}
$$ 
Given $\Lambda, \Lambda'$ two set-partitions of $\{1,2,\ldots, r\}$, we write $\Lambda\subseteq \Lambda'$ 
if  $a \sim_\Lambda b$ implies $a \sim_{\Lambda'} b$  for any $a,b \in \{1,\dots, r\}$.  
We let ${\sf F}^r$ denote the  set consisting of all pairs 
$(\Lambda,\Lambda')$ of set-partitions of $\{1,\dots,r\}$ such that  $\Lambda\subseteq \Lambda'$.
 
We equip ${\sf F}^r$  with a partial ordering $\subseteq$ as follows:  
given 
$(\Lambda, \Lambda'), (\Gamma, \Gamma')\in {\sf F}^r$, we write 
$(\Lambda, \Lambda')
\subseteq
(\Gamma, \Gamma') $ if 
$\Lambda\subseteq \Gamma $   
and 
$\Lambda'\subseteq \Gamma' $.  In this case we say that $
(\Lambda, \Lambda')$ is a refinement of $(\Gamma, \Gamma')$ or that 
$(\Gamma, \Gamma')$  is a coarsening of $
(\Lambda, \Lambda')$. 
Given $(\Lambda,\Lambda') \in {\sf F}^r$ we associate a diagram,  $[\Lambda,\Lambda']$, in the obvious fashion, recording $\Lambda$ as above then grouping together in a `bubble' those parts of $\Lambda$ which together form a part of $\Lambda'$.  
For example if 
$$(\Lambda,\Lambda')=\{\{1, 2, 4\}, \{3\}, \{5, 7, 8 \}, \{6,9\}\}, 
\{\{1, 2, 3, 4\},  \{5, 6, 7, 8,9\}\} $$
then clearly  $\Lambda\subseteq \Lambda'$ and we represent this pair diagrammatically as follows. 
$$[\Lambda,\Lambda']=\begin{minipage}{12cm}
\scalefont{0.7}
 \begin{tikzpicture}[scale=0.7]
    \draw  (0,3.5) arc (180:360:1 and 0);
      \draw  (2,3.5) arc (180:360:2 and .5);
   \draw  (12,3.5) arc (180:360:1 and 0);
      \draw  (8,3.5) arc (180:360:2 and .5);
        \draw  (10,3.5) arc (180:360:3 and -.5);

  \foreach \x in {0,2,...,16}
        \fill[white](\x,3.5) circle (10pt);   
      \draw (4,3.5) node {$ \encircle{3}  $};   
                  \draw (2,3.5) node {$\encircle{2}$}; 
                           \draw (0,3.5) node {$\encircle{1}$}; 
         \draw (6,3.5) node {$\encircle{4}$};   
                  \draw (8,3.5) node {$\encircle{5}$};   
                           \draw (10,3.5) node {$\encircle{6}$};    
         \draw (12,3.5) node {$\encircle{7}$};       \draw (14,3.5) node {$\encircle{8}$}; 
         \draw (16,3.5) node {$\encircle{9}$}; 
    \draw[rounded corners=15pt]
  (-0.75,2.75) rectangle ++(7.5,1.5);
   \draw[rounded corners=15pt]
  (7.25,2.75) rectangle ++(9.5,1.5);      \end{tikzpicture} \end{minipage}
$$

\begin{rmk}\label{usefulconvention}
Continuing with the  convention of \cref{sec2}, we order the the subsets in $\Lambda = \{\Lambda_1,
\dots, \Lambda_l\}$ by \emph{increasing minima}, so that
\[
 1 =   \min \Lambda_1 < \min \Lambda_2 < \cdots < \min \Lambda_{l-1} <
   \min \Lambda_l
  \leq  r.
\]

\end{rmk}


\begin{figure} $$\scalefont{0.7}\begin{tikzpicture}[scale=0.64]

\draw(-1,0)--(-10,-3); 
\draw(-1,0)--(-1,-3); 
\draw(-1,0)--(8,-3); 

\draw(-10,-6)--(-10,-3); 
\draw(-4,-6)--(-1,-3); 
\draw(8,-6)--(8,-3); 

\draw(2,-6)--(-10,-3); 
\draw(2,-6)--(-1,-3); 
\draw(2,-6)--(8,-3);

\draw(2,-6)--(-10,-9); 
\draw(2,-6)--(-4,-9); 
\draw(2,-6)--(8,-9);

\draw(-10,-6)--(-10,-9); 
\draw(-4,-6)--(-4,-9); 
\draw(8,-6)--(8,-9); 

\draw(-10,-9)--(-1,-12); 
\draw(-4,-9)--(-1,-12); 
\draw(8,-9)--(-1,-12); 

\draw[fill,white](-10,-6) circle (5pt);

\draw[fill,white](-1,0)  circle (18pt);  

\draw[fill,white](-10,-3)  circle (18pt); 
\draw[fill,white](-1,-3)  circle (18pt);    
\draw[fill,white](8,-3)  circle (18pt);

\draw[fill,white](-4,-6)  circle (18pt);  
\draw[fill,white](2,-6)  circle (18pt);    
\draw[fill,white](8,-6)  circle (18pt);

\draw[fill,white](-10,-9)  circle (18pt);  
\draw[fill,white](-1,-9)  circle (18pt);     
\draw[fill,white](8,-9)  circle (18pt);      


\path(-1,-12.2) coordinate (origin);  
      \draw[fill=white]   (origin)  circle (0.7);
    \foreach \x in {-2,0,2}
        \fill[white](origin)--++(\x,0) circle (10pt);   
                   \path (origin)--++(0,0) node {$\encircle{2}$}; 
                           \path (origin)--++(-2,0) node {$\encircle{1}$}; 
                                \path (origin)--++(2,0) node {$\encircle{3}$}; 
       \path   (origin)--++(-2,0) coordinate (origin3);  
       \path   (origin)--++(2,0) coordinate (origin4);
 
   \draw    (origin3)  circle (0.7);
   \draw    (origin4)  circle (0.7);

\path(-10,-9) coordinate (origin);   \path   (origin)--++(-2.5,-0.75) coordinate (origin2);  \fill[fill=white,rounded corners=10pt]
  (origin2)  rectangle ++(5,1.5);

    \foreach \x in {-2,0,2}
        \fill[white](origin)--++(\x,0) circle (10pt);   
                   \path (origin)--++(0,0) node {$\encircle{2}$}; 
                           \path (origin)--++(-2,0) node {$\encircle{1}$}; 
                                \path (origin)--++(2,0) node {$\encircle{3}$}; 
    \path   (origin)--++(2.5,-0.75) coordinate (origin2);
     \path   (origin)--++(-2,0) coordinate (origin3);
       \draw[rounded corners=10pt]
  (origin2)  rectangle ++(-3,1.5);

   \draw 
  (origin3)  circle (0.7);

\path(-4,-9) coordinate (origin);  
 \path   (origin)--++(-2.5,-0.75) coordinate (origin2);


 \draw[draw,use Hobby shortcut,closed=true,fill=white]
 (-6.5,-9)..(-6.5+.25,-9+.3)..(-6.5+1.8,-0.65-9)..(-6.5+5-1.8,-0.65-9)..(-6.5+4.75,.3-9)..
 (-6.5+5,0-9)
 ..(-6.5+4.5,-1-9)..(-6.5+2.5,-1.1-9)..(-6.5+0.5,-1-9)..(-6.5,0-9);

       \draw[rounded corners=10pt,fill=white]
  (origin)  circle (0.7cm);

    \foreach \x in {-2,0,2}
        \fill[white](origin)--++(\x,0) circle (10pt);   
                   \path (origin)--++(0,0) node {$\encircle{2}$}; 
                           \path (origin)--++(-2,0) node {$\encircle{1}$}; 
                                \path (origin)--++(2,0) node {$\encircle{3}$}; 
    \path   (origin)--++(-2.5,-0.75) coordinate (origin2);
 

\path(8,-9)  coordinate (origin);   \path   (origin)--++(-2.5,-0.75) coordinate (origin2);
       \fill[fill=white,rounded corners=10pt]
  (origin2)  rectangle ++(5,1.5);

    \foreach \x in {-2,0,2}
        \fill[white](origin)--++(\x,0) circle (10pt);   
                   \path (origin)--++(0,0) node {$\encircle{2}$}; 
                           \path (origin)--++(-2,0) node {$\encircle{1}$}; 
                                \path (origin)--++(2,0) node {$\encircle{3}$}; 
    \path   (origin)--++(-2.5,-0.75) coordinate (origin2);
     \path   (origin)--++(2,0) coordinate (origin3);
       \draw[rounded corners=10pt]
  (origin2)  rectangle ++(3,1.5);

   \draw 
  (origin3)  circle (0.7);

\path(-10,-6)  coordinate (origin);   \path   (origin)--++(-2.5,-0.75) coordinate (origin2);
       \fill[fill=white,rounded corners=10pt]
  (origin2)  rectangle ++(5,1.5);

  \draw(origin)--++(2,0);
    \foreach \x in {-2,0,2}
        \fill[white](origin)--++(\x,0) circle (10pt);   
                   \path (origin)--++(0,0) node {$\encircle{2}$}; 
                           \path (origin)--++(-2,0) node {$\encircle{1}$}; 
                                \path (origin)--++(2,0) node {$\encircle{3}$}; 
    \path   (origin)--++(-2.5,-0.75) coordinate (origin2);
  \path   (origin)--++(2.5,-0.75) coordinate (origin2);
     \path   (origin)--++(-2,0) coordinate (origin3);
       \draw[rounded corners=10pt]
  (origin2)  rectangle ++(-3,1.5);

   \draw 
  (origin3)  circle (0.7);

\path(-4,-6) coordinate (origin);   \path   (origin)--++(-2.5,-0.75) coordinate (origin2);


 \draw[draw,use Hobby shortcut,closed=true,fill=white]
 (-6.5,-5.9)..(-6.5+.25,-5.9+.3)..(-6.5+1.8,-0.65-5.9)..(-6.5+5-1.8,-0.65-5.9)..(-6.5+4.75,.3-5.9)..
 (-6.5+5,0-5.9)
 ..(-6.5+4.5,-1-5.9)..(-6.5+2.5,-1.1-5.9)..(-6.5+0.5,-1-5.9)..(-6.5,0-5.9);

  \path(origin)--++(-2,0) coordinate (origin2);
    \path(origin)--++(2,0) coordinate (origin3);
  \draw(origin2) to  [out=-50,in=-130]  (origin3);   
    \foreach \x in {-2,0,2}
        \fill[white](origin)--++(\x,0) circle (10pt);   
        \fill[white](origin)--++(0,-0.5) circle (10pt);           
                   \path (origin)--++(0,0) node {$\encircle{2}$}; 
                           \path (origin)--++(-2,0) node {$\encircle{1}$}; 
                                \path (origin)--++(2,0) node {$\encircle{3}$}; 
    \path   (origin)--++(-2.5,-0.75) coordinate (origin2);
       \draw[rounded corners=10pt]
  (origin)  circle (0.7cm);

 \draw[draw,use Hobby shortcut,closed=true]
 (-6.5,-5.9)..(-6.5+.25,-5.9+.3)..(-6.5+1.8,-0.65-5.9)..(-6.5+5-1.8,-0.65-5.9)..(-6.5+4.75,.3-5.9)..
 (-6.5+5,0-5.9)
 ..(-6.5+4.5,-1-5.9)..(-6.5+2.5,-1.1-5.9)..(-6.5+0.5,-1-5.9)..(-6.5,0-5.9);



\path(8,-6) coordinate (origin);   \path   (origin)--++(-2.5,-0.75) coordinate (origin2);
       \fill[fill=white,rounded corners=10pt]
  (origin2)  rectangle ++(5,1.5);

  \draw(origin)--++(-2,0);
    \foreach \x in {-2,0,2}
        \fill[white](origin)--++(\x,0) circle (10pt);   
                   \path (origin)--++(0,0) node {$\encircle{2}$}; 
                           \path (origin)--++(-2,0) node {$\encircle{1}$}; 
                                \path (origin)--++(2,0) node {$\encircle{3}$}; 
  \path   (origin)--++(-2.5,-0.75) coordinate (origin2);
     \path   (origin)--++(2,0) coordinate (origin3);
       \draw[rounded corners=10pt]
  (origin2)  rectangle ++(3,1.5);

   \draw 
  (origin3)  circle (0.7);

\path(2,-6) coordinate (origin);   \path   (origin)--++(-2.5,-0.75) coordinate (origin2);
       \draw[fill=white,rounded corners=10pt]
  (origin2)  rectangle ++(5,1.5);

    \foreach \x in {-2,0,2}
        \fill[white](origin)--++(\x,0) circle (10pt);   
                   \path (origin)--++(0,0) node {$\encircle{2}$}; 
                           \path (origin)--++(-2,0) node {$\encircle{1}$}; 
                                \path (origin)--++(2,0) node {$\encircle{3}$}; 
    \path   (origin)--++(-2.5,-0.75) coordinate (origin2);
       \draw[rounded corners=10pt]
  (origin2)  rectangle ++(5,1.5);

\path(-10,-3) coordinate (origin);   \path   (origin)--++(-2.5,-0.75) coordinate (origin2);
       \draw[fill=white,rounded corners=10pt]
  (origin2)  rectangle ++(5,1.5);

  \draw(origin)--++(2,0);
    \foreach \x in {-2,0,2}
        \fill[white](origin)--++(\x,0) circle (10pt);   
                   \path (origin)--++(0,0) node {$\encircle{2}$}; 
                           \path (origin)--++(-2,0) node {$\encircle{1}$}; 
                                \path (origin)--++(2,0) node {$\encircle{3}$}; 
    \path   (origin)--++(-2.5,-0.75) coordinate (origin2);

\path(-1,-3) coordinate (origin);   \path   (origin)--++(-2.5,-0.75) coordinate (origin2);
       \draw[fill=white,rounded corners=10pt]
  (origin2)  rectangle ++(5,1.5);

  \path(origin)--++(-2,0) coordinate (origin2);
    \path(origin)--++(2,0) coordinate (origin3);
  \draw(origin2) to  [out=-30,in=-150]  (origin3);   
    \foreach \x in {-2,0,2}
        \fill[white](origin)--++(\x,0) circle (10pt);   
                   \path (origin)--++(0,0) node {$\encircle{2}$}; 
                           \path (origin)--++(-2,0) node {$\encircle{1}$}; 
                                \path (origin)--++(2,0) node {$\encircle{3}$}; 
    \path   (origin)--++(-2.5,-0.75) coordinate (origin2);
       \draw[rounded corners=10pt]
  (origin2)  rectangle ++(5,1.5);

\path(8,-3) coordinate (origin);   \path   (origin)--++(-2.5,-0.75) coordinate (origin2);
       \draw[fill=white,rounded corners=10pt]
  (origin2)  rectangle ++(5,1.5);

  \draw(origin)--++(-2,0);
    \foreach \x in {-2,0,2}
        \fill[white](origin)--++(\x,0) circle (10pt);   
                   \path (origin)--++(0,0) node {$\encircle{2}$}; 
                           \path (origin)--++(-2,0) node {$\encircle{1}$}; 
                                \path (origin)--++(2,0) node {$\encircle{3}$}; 
    \path   (origin)--++(-2.5,-0.75) coordinate (origin2);
       \draw[rounded corners=10pt]
  (origin2)  rectangle ++(5,1.5);

\path(-1,0) coordinate (origin);   \path   (origin)--++(-2.5,-0.75) coordinate (origin2);
       \draw[fill=white,rounded corners=10pt]
  (origin2)  rectangle ++(5,1.5);

  \draw(origin)--++(-2,0)--++(4,0);
    \foreach \x in {-2,0,2}
        \fill[white](origin)--++(\x,0) circle (10pt);   
                   \path (origin)--++(0,0) node {$\encircle{2}$}; 
                           \path (origin)--++(-2,0) node {$\encircle{1}$}; 
                                \path (origin)--++(2,0) node {$\encircle{3}$}; 
    \path   (origin)--++(-2.5,-0.75) coordinate (origin2);
       \draw[rounded corners=10pt]
  (origin2)  rectangle ++(5,1.5);
   
 \end{tikzpicture}
$$
\caption{The poset ${\sf F}^3$.  }
\label{r=3ex}
\end{figure}


   \begin{defn}
For $m,n \in \N$, we let    ${\sf F}^r_{m,n}\subseteq  {\sf F}^r$ denote the sub-poset consisting of the   elements $(\Lambda,\Lambda')$ such that 
     $\ell(\Lambda') \le n$ and such that the number of blocks of $\Lambda$ contained within any single block of $\Lambda'$ is at most $m$.  
\end{defn}

\begin{eg}
\label{newerererer} The subposet 
 ${\sf F}^{3}_{2,3}\subseteq {\sf F}^{3}$ contains all the elements shown in  \cref{r=3ex} except the leftmost diagram in \cref{newfig}.  
 The subposet 
 ${\sf F}^{3}_{3,2}\subseteq {\sf F}^{3}$ contains all the elements shown in  \cref{r=3ex} except the rightmost diagram in 
\cref{newfig}.

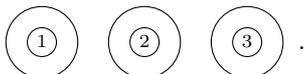
\begin{figure}[ht!]
$$
\begin{minipage}{3.5cm}\scalefont{0.7}\begin{tikzpicture}[scale=0.675]
\path(2,-6) coordinate (origin);   \path   (origin)--++(-2.5,-0.75) coordinate (origin2);
       \draw[fill=white,rounded corners=15pt]
  (origin2)  rectangle ++(5,1.5);

     \foreach \x in {-2,0,2}
        \fill[white](origin)--++(\x,0) circle (10pt);   
                   \path (origin)--++(0,0) node {$\encircle{2}$}; 
                           \path (origin)--++(-2,0) node {$\encircle{1}$}; 
                                \path (origin)--++(2,0) node {$\encircle{3}$}; 
    \path   (origin)--++(-2.5,-0.75) coordinate (origin2);
       \draw[rounded corners=15pt]
  (origin2)  rectangle ++(5,1.5);
     
 \end{tikzpicture}\end{minipage} 
\qquad \qquad   \qquad 
\begin{minipage}{3.85cm}\scalefont{0.7}\begin{tikzpicture}[scale=0.675]
\path(-1,-12.2) coordinate (origin);  

    \foreach \x in {-2,0,2}
        \fill[white](origin)--++(\x,0) circle (10pt);   
                   \path (origin)--++(0,0) node {$\encircle{2}$}; 
                           \path (origin)--++(-2,0) node {$\encircle{1}$}; 
                                \path (origin)--++(2,0) node {$\encircle{3}$}; 
       \path   (origin)--++(-2,0) coordinate (origin3);  \path   (origin)--++(2,0) coordinate (origin4);
      \draw   (origin)  circle (0.7);
   \draw    (origin3)  circle (0.7);
   \draw    (origin4)  circle (0.7);
 \end{tikzpicture} \end{minipage} 
$$
\caption{Two diagrams discussed in 
\cref{newerererer} which do not belong to certain subposets of 
${\sf F}^{3}$ (cross reference with \cref{r=3ex}).
}
\label{newfig}
\end{figure}
\end{eg}

  \section{Schur--Weyl duality} \label{sec3}

  Classical Schur--Weyl duality is the relationship between the general linear
and symmetric groups over tensor space. 
 The symmetric group acts on the right by permuting the factors. The general linear group  acts on the left by matrix multiplication on each factor. These two actions commute and 
each generates the full centraliser of the other.  
We can restrict the action of the general linear group to the subgroup of all  permutation matrices  and ask what algebra appears on the other side of the duality? The answer is the partition algebra.
Through this duality,   the partition algebra governs the stability phenomena of  the symmetric group   \cite{BDO15}.   For the purposes of this paper, we shall be interested in how this allows us to understand certain plethysm coefficients  via their  stabilities in this limit.

The purpose of this paper is  to study the  decomposition of the  Foulkes module $\ind_{\Sym _m\wr \Sym _n}^{\Sym _{mn}} \CC$.  
We shall attempt to do this via Schur--Weyl duality with the partition algebra. 
 In order to see the Foulkes module within the setting of Schur--Weyl duality, we must consider the left action of $\Sym _{mn}$ 
 and the right action of $P_r(mn)$ 
 on     $(\mathbb{C}^{mn})^{\otimes r}$ (defined below). Here we have specialised  the parameter of the partition algebra to  $\delta=mn \in \NN$.

  Let $I(mn,r) = \{1, \dots, mn\}^r$ be the set of multi-indices. For a
given multi-index $i = (i_1, \dots, i_r) \in I(mn,r)$, we put $\elemvec_i =
\elemvec_{i_1} \otimes \cdots \otimes \elemvec_{i_r}$. Then $\{ \elemvec_i : i \in I(mn,r)
\}$ is a basis of tensor space $(\Bbbk^{mn})^{\otimes r}$ over $\Bbbk$.
 The action of the Weyl group $\W_{mn}$  on $(\Bbbk^{mn})^{\otimes r}$ is simply the restriction of the
diagonal action of $\GL_{mn}(\Bbbk)$.
 In particular, 
\begin{equation} \label{eqn:tensor_action}
 \sigma \cdot (\elemvec_{i_1} \otimes \cdots \elemvec_{i_r}) = \elemvec_{\sigma(i_1)} \otimes \cdots \elemvec_{\sigma(i_r)}
\end{equation} 
for any $\sigma \in \W_{mn}$ and $i = (i_1, \dots, i_r) \in I(mn,r)$.

Let $d \in \Pmn$ be a partition diagram. Let $\delta_{a,b}$ be the
usual Kronecker delta symbol, defined to be $1$ if $a=b$ and $0$
otherwise. Then the (right) action of $d$ on the basis element $\elemvec_i$ of tensor space is given by
the matrix $\Psi(d)$ whose entry
$[\Psi(d)]^{i_1, \dots, i_r}_{i_{\overline{1}},   \dots, i_{\overline{r}}}$ in row $(i_1, \dots, i_r)$ and column $(i_{\overline{1}}, \dots, i_{\overline{r}})$ is given by
\begin{equation}\label{otheraction}
 [\Psi(d)]^{i_1, \dots, i_r}_{i_{\overline{1}}, \dots, i_{\overline{r}}} 
=   \prod \delta_{i_s, i_t}
\end{equation}
where the product is taken over all pairs $s,t$ in $\{1,\dots, r, \overline{1},
\dots, \overline{r}\}$ which are connected by an edge in $d$ 
(see, for example, \cite[Equation~(3.8)]{HalvRam}).  
  It is
easily checked that the linear extension of the rule $d \mapsto
\Psi(d)$ defines a representation $\Psi: \Pmn
\to \End_\Bbbk((\Bbbk^{mn})^{\otimes r})$.  
 To summarise,  
writing $\Phi$ for the map induced by the action in \cref{eqn:tensor_action},
we have actions of the symmetric group and   partition algebra on tensor space as follows
\begin{equation}
  \Bbbk \W_{mn}   \xrightarrow{\ \Phi  \ }
   \End_\Bbbk((\Bbbk^{mn})^{\otimes r}) 
\xleftarrow{\ \Psi \ }  \Pmn.
\end{equation}



\begin{thm}[\cite{jones, marbook}]\label{mrjones} 
  In the situation outlined above, the image of each representation is
  equal to the full centraliser algebra for the other action. More
  precisely, we have equalities
  \[
  \Phi (\Bbbk \W_{mn}) = \End_{\Pmn}((\Bbbk^{mn})^{\otimes r}), \qquad 
  \Psi(\Pmn) = \End_{\W_{mn}}((\Bbbk^{mn})^{\otimes r}).
  \]
  As a $ (\CC \W_{mn},\Pmn)$-bimodule, the tensor space decomposes as 
  $$
  (\Bbbk^{mn})^{\otimes r} \cong \bigoplus  {\sf S}(\lambda_{[mn]})\otimes L_{r, mn}(\lambda)
  $$
  where the sum is over all partitions $\lambda_{[mn]}$ of $mn$  such that $|\lambda|\leq r$.  

\end{thm}

Here, if $M$ is a left $\W_{mn}$-module then   
  $$
 \mathcal{F}_r(M)=\Hom_{\W_{mn}}(M, (\Bbbk^{mn})^{\otimes r})
  $$
  carries the structure of a  right $P_r(mn)$-module. Conversely, if $N$ is a  right $\Pmn$-module, 
   then    
  $$
 \Hom_{\Pmn}( N, (\Bbbk^{mn})^{\otimes r})
  $$
  carries the structure of a left $\W_{mn} $-module.  In particular, the theorem shows that  $\mathcal{F}_r( {\sf S}(\lambda_{[mn]}) )  \cong  L_{r, mn}(\lambda)$.

Now, following~\cite[Section 4.1]{jk}, we consider
$$\W_m \wr \W_n =\{ (\sigma_1,\sigma_2,\ldots, \sigma_n ; \pi) : \sigma_i \in  \Sym _m, i=1,\ldots, n, \pi \in \Sym _n\},$$
which we identify with a subgroup of $ \Sym _{mn}$ via the embedding
\begin{equation}\label{eqn:wreath_embedding}
(\sigma_1,\sigma_2,\ldots, \sigma_n ; \pi ) \mapsto \left( \begin{array}{c} (j-1)m+i \\ (\pi(j)-1)m+\sigma_{\pi(j)}(i) \end{array}\right)_{i=1,\ldots, m, j=1\ldots, n}.
\end{equation}
   The action~(\ref{eqn:tensor_action}) of $\W_{mn}$ on tensor space restricts to an action of $\W_m \wr \W_n$ (using the  identification~(\ref{eqn:wreath_embedding})).     
 Having chosen our wreath product subgroup in the fashion above, we let this guide our choice of a new labelling set for the basis of tensor space   as follows.  For  $1\leq i \leq m$ and $1\leq j \leq n$, we set
   $$v^j_i = \elemvec_{(j-1)m+i} ,$$ 
 and we note that    
 \begin{equation}
 \label{eqn:action}
(\sigma_1,\sigma_2,\dots, \sigma_n ; \pi)    (v_{i_1}^{j_1}\otimes     v_{i_2}^{j_2} \otimes \dots \otimes     v_{i_r}^{j_r})
=
 v^{\pi(j_1)}_{\sigma_{\pi(j_1)}(i_1)}
 \otimes 
  v^{\pi(j_2)}_{\sigma_{\pi(j_2)}(i_2)}
  \otimes 
  \dots
  \otimes
   v^{\pi(j_r)}_{\sigma_{\pi(j_r)}(i_r)}.  
  \end{equation} 
Using Schur--Weyl duality, we now apply $\mathcal{F}_r$ to the Foulkes module to define a   $P_r(mn)$-module. This module's decomposition into simple constituents will be  governed  by the plethysm coefficients.  
This will allow  us to study plethysm coefficients using the tools from  the representation theory of partition algebras. 
\begin{defn} 
  We say that the basis vector
  $$v = v^{j_1}_{i_1}
  \otimes 
  v^{j_2}_{i_2}
  \otimes 
  \dots
  \otimes 
  v^{j_r}_{i_r}
  \in (\Bbbk^{mn})^{\otimes r}$$ has
 \emph{value-type} $(\Lambda, \Lambda')$ if 
    $k \sim_{\Lambda'}l$ if and only if $j_k=j_l$ and 
   $k \sim_{\Lambda}l$ if and only if $j_k=j_l$ and $i_k=i_l$. We write  ${\rm val}(v)=(\Lambda, \Lambda')$.
Observe that $\Lambda, \Lambda'$ are set-partitions of $\{1,2,\dots, r\}$ with $\Lambda \subseteq \Lambda'$ and, since there are at most $m$ possible subscripts and $n$ possible superscripts, ${\rm val}(v) \in   {\sf F}^r_{m,n}$.    
  Given $(\Lambda, \Lambda') \in {\sf F}^r_{m,n}$, we let $v_ {(\Lambda,\Lambda')} $ denote the vector 
$$v_ {(\Lambda,\Lambda')} =\sum_{{\rm val}(v)=(\Lambda,\Lambda') } v, $$
where the sum runs over all basis vectors of $ (\Bbbk^{mn})^{\otimes r}$ with value-type $(\Lambda,\Lambda')$.
  \end{defn}

   \begin{eg}  For example, the basis vector
$ v =v_2^1 \otimes v_1^1 \otimes v_1^1 \otimes v_3^2 \otimes v_2^3 \otimes v_3^2 \otimes v_3^3$ 
has 
$${\rm val}(v) = ( \{\{1\}, \{2,3\}, \{4,6\},\{5\},\{7\}\}, \{\{1,2,3\}, \{4,6\}, \{5,7\}\}).$$
To obtain $\Lambda'= \{\{1,2,3\}, \{4,6\}, \{5,7\}\}$  note that 
the superscripts match in positions 1,2,3 and they match in positions 4 and 6 and they also match in positions 5 and 7.
 Although the subscripts match in positions $1$ and $5$, however the superscripts do not match and so $1\nsim_{\Lambda} 5$.  
  \end{eg}

 \begin{eg} Let  $m=n=r=2$ and $ (\Lambda,\Lambda') =(\{\{1\},\{2\}\}, \{\{1,2\} \})$ then
 $$
  v_ {(\Lambda,\Lambda')} = 
  v_{2}^{1}\otimes v_{1}^{1}
  +
  v_{1}^{1}\otimes v_{2}^{1}
  +
  v_{2}^{2}\otimes v_{1}^{2}
  +
  v_{1}^{2}\otimes v_{2}^{2}.
 $$
 \end{eg}
\begin{eg} \label{ex:4,5}Let $m=4$ and $n=r=5$ and $$ {(\Lambda,\Lambda')}= 
(  \{\{1,2,4\},\{3\},\{5\}\}, \{\{1,2,3,4\},\{5\}\} ).$$  Then
 $$
 v_ {(\Lambda,\Lambda')}= \sum_{
 \begin{subarray}c
 1\leq i_1, i_2,i_3 \leq 4\\
   1\leq j_1, j_2 \leq 5\\
 i_1 \neq i_2, j_1\neq j_2 \\  
 \end{subarray}
 }
 v_{i_1}^{j_1}\otimes  v_{i_1}^{j_1}\otimes  v_{i_2}^{j_1}\otimes  v_{i_1}^{j_1}
 \otimes  v_{i_3}^{j_2}.
 $$
\end{eg}

Consider the action of the group $\W_m \wr \W_n$ on $(\Bbbk^{mn})^{\otimes r}$.

\begin{lem}\label{lem:orbit_sum}
The $\W_m \wr \W_n$-orbit of a basis vector $v$ of $(\Bbbk^{mn})^{\otimes r}$ consists of precisely those basis vectors having value-type equal to ${\rm val}(v)$.

  \end{lem}
  
  \begin{proof}
This is follows from Equation~(\ref{eqn:action}).
%
%
\end{proof}

 \begin{cor} \label{phi-basis}   For each $ {(\Lambda,\Lambda')}\in {\sf F}^r_{m,n}$, define an element 
$$\varphi_ {(\Lambda,\Lambda')}\in  \Hom_{\W_{mn}}(\W_{mn}\otimes_{\W_m\wr\W_n}\CC, (\CC^{mn})^{\otimes r})
  $$
 by setting 
 $$\varphi_ {(\Lambda,\Lambda')}(\sigma \otimes 1 )= \sigma v_ {(\Lambda,\Lambda')}$$
    for any $\sigma \in \Sym _{mn}$.   The set 
    $$\{\varphi_{ {(\Lambda,\Lambda')}}\mid  {(\Lambda,\Lambda')} \in {\sf F}_{m,n}^r\}$$
    is a basis for the right $\Pmn$-module 
$     \Hom_{\W_{mn}}(\W_{mn}\otimes_{\W_m\wr\W_n}\CC, (\CC^{mn})^{\otimes r})$.

    \end{cor}
\begin{proof}    
The statement  follows by  Frobenius reciprocity and Lemma~\ref{lem:orbit_sum}.  
     \end{proof}

\begin{defn}
  We refer to the (right) $\Pmn$-module 
$$     \mathcal{F}_r(\W_{mn}\otimes_{\W_m\wr\W_n}\CC)= \Hom_{\W_{mn}}(\W_{mn}\otimes_{\W_m\wr\W_n}\CC, (\CC^{mn})^{\otimes r})$$ 
 as the \emph{Foulkes module} for $\Pmn$.
 \end{defn}
      We give a second basis of the  Foulkes module for $\Pmn$, $\{ \overline{\varphi}_ {(\Lambda,\Lambda')} \, : \, {(\Lambda,\Lambda')}\in {\sf F}^r_{m,n}\}$, by setting 
$$  
   \overline{\varphi}_ {(\Lambda,\Lambda')}= \sum_{(\Gamma,\Gamma'): (\Lambda,\Lambda')\subseteq (\Gamma,\Gamma')} \varphi_ {(\Gamma,\Gamma')}. 
$$
The element $\overline{\varphi}_ {(\Lambda,\Lambda')}$ can be defined for any $(\Lambda,\Lambda')\in {\sf F}^r$, sending the 
the generator $1_{\Sym _{mn}} \otimes 1$ to
\begin{equation}\label{eqn:phi_bar}
 \sum_{\begin{subarray}c
 1\leq i_S \leq m, \forall S \in \Gamma, \\
   1\leq j_\Sigma \leq n, \forall \Sigma \in \Gamma',\\  
   i_a=i_S, j_a=j_\Sigma, \forall  a \in S \subseteq \Sigma \\
 \end{subarray}
  } v^{j_1}_{i_1} \otimes v^{j_2}_{i_2} \otimes  \cdots v^{j_r}_{i_r}.
\end{equation}
 For example, if $(\Lambda,\Lambda')$ is as in Example~\ref{ex:4,5} then 
 $$\overline{\varphi}_ {(\Lambda,\Lambda')}(1_{\W_{mn}} \otimes 1)=\sum_{
  \begin{subarray}c
 1\leq i,i',i'' \leq 4\\
   1\leq j,j' \leq 5\\  
 \end{subarray}} v^{j}_{i} \otimes v^{j}_{i} \otimes v^{j}_{i'}\otimes v^{j}_{i}\otimes v^{j'}_{i''}.
 $$


\section{The   stable Foulkes module}\label{sec:diagrammatic}
In this section we factorise the 
partition algebra parameter  $\delta=\delta_1  \delta_2$.   
In practice, we will often specialise $\delta_1=m$ and $\delta_2=n \in \NN$, however it  will be useful to be able to specialise from generic parameters. 
%
%
%
 We define the \emph{stable Foulkes module} to be the complex vector space
 $$  \Fdelta = {\rm Span}_{\mathbb{C}}\{[\Lambda, \Lambda'] \mid (\Lambda, \Lambda')\in {\sf F}^r\},$$
equipped with the following right  action of $ \Pdelta$. For $d$ an $(r,r)$-partition diagram, 
$$[\Lambda, \Lambda'] d = \delta_1^{t_1} \delta_1^{t_2}  [\Gamma, \Gamma'],$$
if
$[\Lambda] d = \delta_1^{t_1}  [\Gamma]$ in the $P_r(\delta_1)$-module $\Delta_{r, \delta_1}(\emptyset)$, and 
$[\Lambda'] d = \delta_2^{t_2}  [\Gamma']$ in the $P_r(\delta_2)$-module $\Delta_{r, \delta_2}(\emptyset)$.
 (Recall from \cref{sec2} that the $P_r(\delta)$-module  $\Delta_{r, \delta}(\emptyset)$  has diagrammatic basis given by all set-partitions of $\{1,2,\ldots,r\}$ with the natural concatenation action.)


\begin{eg}\label{egoffiltration}
Let $r=2$.  The stable Foulkes module  $\mathbb{F}^2(\delta_1, \delta_2)$  is 3-dimensional with basis given by the following three diagrams:
$$\scalefont{0.8}
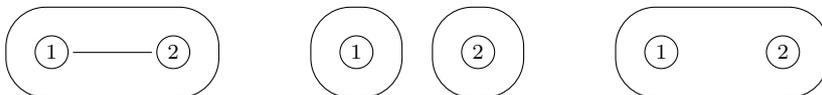
 \begin{tikzpicture}[scale=0.7]
    \draw  (0,3.5) arc (180:360:1 and 0);
   \foreach \x in {0,2}
        \fill[white](\x,3.5) circle (10pt);   
                   \draw (2,3.5) node {$\encircle{2}$}; 
                           \draw (0,3.5) node {$\encircle{1}$}; 
     \draw[rounded corners=15pt]
  (-0.75,2.75) rectangle ++(3.5,1.5);
   
      \end{tikzpicture} 
 \qquad \qquad \quad \quad 
\begin{tikzpicture}[scale=0.7]
   \foreach \x in {0,2}
        \fill[white](\x,3.5) circle (10pt);   
                   \draw (2,3.5) node {$\encircle{2}$}; 
                           \draw (0,3.5) node {$\encircle{1}$}; 
     \draw[rounded corners=15pt]
  (-0.75,2.75) rectangle ++(1.5,1.5);
     \draw[rounded corners=15pt]
  (-0.75+2,2.75) rectangle ++(1.5,1.5);

      \end{tikzpicture} \qquad \qquad \quad \quad 
\begin{tikzpicture}[scale=0.7]
   \foreach \x in {0,2}
        \fill[white](\x,3.5) circle (10pt);   
                   \draw (2,3.5) node {$\encircle{2}$}; 
                           \draw (0,3.5) node {$\encircle{1}$}; 
     \draw[rounded corners=15pt]
  (-0.75,2.75) rectangle ++(3.5,1.5);
   
      \end{tikzpicture} 
$$
The generators of $P_2(\delta_1 \delta_2)$ act as follows:
$$
\pgen _1 \mapsto 
\left(\begin{array}{ccc}0 & 0 & 0 \\1 & \delta_1\delta_2 & \delta_1 \\0 & 0 & 0\end{array}\right) ,
\qquad 
\pgen _{1,2} \mapsto 
\left(\begin{array}{ccc}1 & 1 & 1   \\0 & 0 & 0 \\0 & 0 & 0 \end{array}\right) ,
\qquad 
\sgen _{1,2} \mapsto 
\left(\begin{array}{ccc}1 & 0 & 0   \\0 & 1 & 0 \\0 & 0 & 1 \end{array}\right) .
$$
From this, one observes that the first two diagrams span a 2-dimensional $P_2(\delta_1  \delta_2)$-submodule isomorphic to $\Delta_{2, \delta_1\delta_2}(\varnothing)$.  The resulting quotient module  is isomorphic to 
$\Delta_{2, \delta_1\delta_2}((2))$.  
\end{eg} 

We  now describe the action of the generators of the partition algebra $\Pdelta$ on the basis of the    stable   Foulkes module.  
We let $ \Lambda=\{S_1, S_2, \dots , S_p\}$ and $ \Lambda' =  \{ \Sigma_1,\Sigma_2,\dots, \Sigma_q \} $ for $1\leq q\leq p \leq r$.  
Recall the notational convention of \cref{usefulconvention}.  
Observe that 
\begin{equation}  \label{p12}[\Lambda,\Lambda']\pgen _{1,2}
=
\begin{cases}  
  [\{S_1, S_2, S_3, \dots S_p\},\{\Sigma_1,\Sigma_2,\dots, \Sigma_q\}] &\text{if }1,2\in S_1 \subseteq  \Sigma_1
\\  
[\{S_1\cup S_2, S_3, \dots S_p\},\{\Sigma_1\cup\Sigma_2,\dots, \Sigma_q\}] &\text{if }1\in S_1 \subseteq  \Sigma_1, 2\in S_2 \subseteq \Sigma_2
\\
  [\{S_1\cup S_2, S_3, \dots S_p\},\{\Sigma_1,\Sigma_2,\dots, \Sigma_q\}] &\text{if }1\in S_1 \subseteq  \Sigma_1, 2\in S_2 \subseteq \Sigma_1
 \end{cases}\end{equation}
and also
\begin{equation}  \label{p1}[\Lambda,\Lambda']\pgen _{1}
=
\begin{cases}
\delta_1 \delta_2 \times [\Lambda,\Lambda']  &\text{if } \{1\}=S_1 = \Sigma_1
\\
  [\{\{1\}, S_1-\{1\}, S_2, S_3, \dots S_p\},\{\{1\}, \Sigma_1-\{1\},\Sigma_2,\dots, \Sigma_q\}] &\text{if }\{1\}\subset S_1 \subseteq   \Sigma_1
\\
\delta_1  \times  [\{\{1\},   S_2, S_3, \dots S_p\},\{\{1\}, \Sigma_1-\{1\},\Sigma_2,\dots, \Sigma_q\}] &\text{if }\{1\}= S_1 \subset    \Sigma_1.
 \end{cases}
 \end{equation}
The generators $\sgen _{i,i+1}$ for $1\leq i <r$ act in the usual fashion by  permuting $\{1,2,\ldots,r\}$.
%
For ease of notation, we do not write these actions out explicitly.  


 
\section{Relating the  Foulkes  and stable Foulkes  modules for the partition algebra} \label{relateS}
We again  specialise the parameters $\delta_1=m$ and $\delta_2=n$ to relate the two right $\Pmn$-modules we introduced in the previous sections: the  stable Foulkes module $\Fmn$ and the 
Foulkes module $\Hom_{\W_{mn}}(\W_{mn}\otimes_{\W_m\wr\W_n}\CC, (\CC^{mn})^{\otimes r})$. 
We shall do this using the  elements $\overline{\varphi}_{(\Lambda, \Lambda')}$, which   were defined in (\ref{eqn:phi_bar}).
 
    \begin{thm} \label{schurfunctor2}
Let $m,n,r \in \mathbb{N}$.  There is a  surjective homomorphism of   $\Pmn$-modules     
$$\Theta:
%
\Fmn \twoheadrightarrow   \Hom_{\W_{mn}}(\W_{mn}\otimes_{\W_m\wr\W_n}\CC, (\CC^{mn})^{\otimes r}),
   $$  
 mapping  $[\Lambda, \Lambda']$ to $\overline{\varphi}_{(\Lambda, \Lambda')}$.  
   The map $\Theta$ is injective if and only if $m,n\geq r$.  
    \end{thm}
    
    \begin{proof} 
We have that  $\{  {\varphi}_ {(\Lambda,\Lambda')} \, : \, {(\Lambda,\Lambda')}\in {\sf F}^r_{m,n}\}$ is a basis of the Foulkes module by  \cref{phi-basis}
and   therefore  $\{ \overline{\varphi}_ {(\Lambda,\Lambda')} \, : \, {(\Lambda,\Lambda')}\in {\sf F}^r_{m,n}\}$ is a basis by unitriangularity.  
Thus, surjectivity of $\Theta$ is clear.
 Injectivity follows if and only if ${\sf F}^r_{m,n}={\sf F}^r$, that is $m,n\geq r$. 
 It remains to check that $\Theta$ is a $P_r(mn)$-modules homomorphism, which is simply an exercise in matching-up the action of the partition algebra generators on the  two modules.   
We write $\overline{v}_{(\Lambda,\Lambda')}$ for the image of the generator $1_{\Sym _{mn}} \otimes 1$  of the Foulkes module under $\overline{\varphi}_ {(\Lambda,\Lambda')}$ from \cref{eqn:phi_bar}.

In $\overline{v}v_{(\Lambda,\Lambda')} \pgen _{1,2}$, the term $v^{j_1}_{i_1} \otimes v^{j_2}_{i_2} \otimes  \cdots v^{j_r}_{i_r}$ is killed if $i_1 \ne i_2$ or $j_1 \ne j_2$. The effect is therefore to merge the  blocks of  $\Lambda$ (respectively $\Lambda'$) containing $1$ and $2$. Compare this with Equation~(\ref{p12}).

Now consider $\pgen _{1}$. Each term $v^{j_1}_{i_1} \otimes v^{j_2}_{i_2} \otimes  \cdots v^{j_r}_{i_r}$ in  $\overline{v}_{(\Lambda,\Lambda')}$ is sent to $\sum_{1\le i \le m, 1\le j \le n} v^{j }_{i } \otimes v^{j_2}_{i_2} \otimes  \cdots v^{j_r}_{i_r}$ in $\overline{v}_{(\Lambda,\Lambda')} \pgen _{1}$.  The effect is to split off a singleton block $\{1\}$ from the blocks of $\Lambda$ and $\Lambda'$ containing $1$, but there is a multiplicity.  If the singleton part $\{1\} \in \Lambda$ and  $\{1\} \in \Lambda'$ then there are $mn$ terms in $\overline{v}_{(\Lambda,\Lambda')}$ making this contribution. If $\{1\}\in \Lambda$ but is not a singleton part of $\Lambda'$ then there are $m$ terms making this contribution. Finally, if $\{1\}$ is not a singleton part of either $\Lambda$ or $\Lambda'$ then the single term contributes. Compare this with Equation~(\ref{p1}). 

The symmetric group generators act in the usual way and thus $\Theta$ is  a $P_r(mn)$-homomorphism.
\end{proof}

\begin{cor}\label{schurfunctor}
Let $m,n,r \in \mathbb{N}$.  We have the following equality of composition multiplicities:
 $$
\left[ \W_{mn}\otimes_{\W_m\wr\W_n}\CC \, : \, S(\lambda_{[mn]}) \right]_{\CC \W_{mn}}
=
\left[  \Fmn /\ker(\Theta) \, : \, L_{r, mn}(\lambda) \right]_{\Pmn},$$ 
for $\lambda_{[mn]}$ a partition of $mn$ such that $|\lambda|\leq r$.  In particular, if $m,n \ge r$ then 
$$
\left[ \W_{mn}\otimes_{\W_m\wr\W_n}\CC \, : \, S(\lambda_{[mn]}) \right]_{\CC \W_{mn}}
=
\left[  \Fmn  \, : \, L_{r, mn}(\lambda) \right]_{\Pmn}.$$
\end{cor}
\begin{proof}
This follows from \cref{mrjones,schurfunctor2}.  
\end{proof}

 We will demonstrate the power of this corollary in the next two sections, where   we use the partition algebra to give elementary algebraic proofs of  results about plethysm coefficients.

 \section{The structure of the  stable  Foulkes module}\label{sec5}
     
     We begin this section by giving an elementary filtration on the stable Foulkes module.  We show that the layers of this filtration are preserved 
     under swapping the parameters; this gives a simple proof of Theorem A.  We then examine these layers of the filtration in more detail; we provide an explicit direct sum  decomposition of these layers   and hence prove Theorem B.  
     
  \subsection{A filtration of the stable  Foulkes module }

In this section we construct a filtration of the stable  Foulkes module $\Fdelta$ as a $\Pdelta$-module (with arbitrary parameters $\delta_1, \delta_2$). This will allow us to deduce that its composition factors
 depend {\em only} on  the product $\delta_1\delta_2 \in \CC$ and   
 do not depend on the distinct parameters $\delta_1,\delta_2$.  
%
 Given    $\Lambda$ a set-partition, recall that  $\sharp (\Lambda)$ denotes the number of blocks in $\Lambda$.  
\begin{defn}
 Given a pair $(\Lambda,\Lambda')$, we set  $\sharp([\Lambda,\Lambda'])= \sharp(\Lambda)-\sharp(\Lambda').$     
For $0\leq k < r$, we let $\mathbb{F}^r_k(\delta_1,\delta_2)$     denote the subspace of $\Fdelta$ with basis 
$   \{	  [\Lambda,\Lambda'] \mid \sharp([\Lambda,\Lambda']) \leq k		\}.$ 
\end{defn}
    For example  for $r=8$ we see that 
$$(\Lambda,\Lambda')=\{\{1, 2, 4\}, \{3\}, \{5, 7, 8 \}, \{6,9\}\}, 
 \{\{1, 2, 3, 4\},  \{5, 6, 7, 8,9\}\}\in \mathbb{F}^9_2(\delta_1, \delta_2) . $$     
 
\begin{thm}
Let $r \in \mathbb{N}$.  Then, as a  $\Pdelta$-module, $\Fdelta$ has a filtration 
$$
0\subset  \mathbb{F}^r_{0}(\delta_1,\delta_2)\subset
\mathbb{F}^r_{1}(\delta_1,\delta_2)
 \subset \dots \subset 
\mathbb{F}^r_{r-1}(\delta_1,\delta_2)
= \Fdelta.
$$
Moreover, all entries in the representing matrices of the generators of $\Pdelta$ on the quotient module 
$$\mathbb{F}^r_k(\delta_1,\delta_2)/ \mathbb{F}^r_{k-1}(\delta_1,\delta_2)$$
for $1\leq k \le r-1$ consist only of zeroes, ones, and the parameter $\delta_1 \delta_2$.  In particular, the entries    depend only on  the product $\delta_1 \delta_2 $ and   
 are independent of the factors $\delta_1,\delta_2$.  
\end{thm}
\begin{proof}
We shall consider the actions of the elements  $\pgen _{1,2}$, $\pgen _1$ and $\sgen _{i,i+1}$ for $1\leq i <r$ in turn.  
 In the   three cases of \cref{p12},  we have that $$\sharp([\Lambda,\Lambda'])- \sharp([\Lambda,\Lambda']\pgen _{1,2})=0,0,1$$ respectively.   
Each entry in the representing matrix is 0 or 1.
Now consider the action of  $\pgen _1$ from \cref{p1}: $[\Lambda,\Lambda']\pgen _{1}$ is a scalar times a diagram $[\Gamma, \Gamma']$ and
 $$\sharp([\Lambda,\Lambda'])- \sharp([\Gamma, \Gamma'])= 0,0,1$$respectively.
Furthermore, note that the parameter $\delta_1$ appears only in the third case of \cref{p1},
 which is precisely the case in which 
 $[\Lambda,\Lambda']\pgen _{1}$ is zero in  the quotient module 
$\mathbb{F}^r_k(\delta_1,\delta_2)/ \mathbb{F}^r_{k-1}(\delta_1,\delta_2)$.
Finally, the elements $\sgen _{i,i+1}$  simply permute  the numbers $\{1,2,\ldots, r\}$ within the  blocks of the set-partition and so, firstly, 
$\sharp([\Lambda,\Lambda'])- \sharp([\Lambda,\Lambda']\sgen _{i,i+1})= 0$ for all for $1\leq i < r$ and, secondly,  the representing matrices of $\sgen _{i,i+1}$ consist only of entries 0 and 1.  
 The result follows.  
\end{proof} 

\begin{eg}
This filtration is constructed  explicitly in \cref{egoffiltration} for $r=2$ and $\delta_1,\delta_2 $ arbitrary.  
\end{eg}

\begin{cor}\label{isomorphism}
Let $ r \in \mathbb{N}$.  There exists an  isomorphism of $\Pdelta$-modules
$$\mathbb{F}^r_k(\delta_1,\delta_2)/ \mathbb{F}^r_{k-1}(\delta_1,\delta_2) \cong \mathbb{F}^r_k(\delta_2, \delta_1)/\mathbb{F}^r_{k-1}(\delta_2, \delta_1)$$
for $1\leq k \le r-1$.
In particular, we have the following equality of composition multiplicities:
 $$
\left[\Fdelta  \, : \, L_{r,\delta_1\delta_2}(\lambda)  \right]_{\Pdelta }
=
\left[\mathbb{F}^r(\delta_2, \delta_1)  \, : \, L_{r,\delta_1\delta_2}(\lambda)  \right]_{\Pdelta },$$
for $|\lambda|\le r$.
\end{cor}
In particular, specialising the parameters $\delta_1\delta_2=m n\in \NN$, we obtain Theorem A of the introduction. 


\subsection{An explicit decomposition of the   stable Foulkes module}

In this section we decompose the   stable  Foulkes module in the case where $\Pdelta$  is semisimple.
 
\begin{defn}\label{depthradical}Let $r\in \mathbb{N}$.  
We define the \emph{depth-radical}   of $\Fdelta$ to be the subspace spanned by the 
pairs $[\Lambda,\Lambda']$ satisfying either of the following two conditions:
\begin{itemize}
\item[$(i)$] The set-partition $\Lambda$ contains a non-singleton block; 
\item[$(ii)$] the set-partition $\Lambda'$ contains a singleton block.  
\end{itemize}
We let ${\sf DR}(\Fdelta)$ denote the  depth-radical    of $\Fdelta$.  
\end{defn}

\begin{eg}\label{exampleee}
For  $r=4$ the module $\mathbb{F}^4(\delta_1, \delta_2)$ is 60-dimensional and 
${\sf DR}(\mathbb{F}^4(\delta_1, \delta_2))$ is 56-dimensional. Rather than list the basis  elements 
of ${\sf DR}(\mathbb{F}^4(\delta_1, \delta_2))$, we instead list the four pairs $(\Lambda,\Lambda')$ which do {\em not} belong to the depth-radical.  These are pictured below. 
$$\scalefont{0.8}
 \begin{tikzpicture}[scale=0.8]
    \foreach \x in {0,2,4,6}
        \fill[white](\x,3.5) circle (10pt);   
                   \draw (2,3.5) node {$\encircle{2}$}; 
                           \draw (0,3.5) node {$\encircle{1}$}; 
                                \draw (4,3.5) node {$\encircle{3}$}; 
                           \draw (6,3.5) node {$\encircle{4}$}; 
     \draw[rounded corners=15pt]
  (-0.5,2.75) rectangle ++(3,1.5);
     \draw[rounded corners=15pt]
  (-0.5+4,2.75) rectangle ++(3,1.5); 
      \end{tikzpicture}\qquad\qquad\qquad
  \begin{tikzpicture}[scale=0.8]
    \foreach \x in {0,2,4,6}
        \fill[white](\x,0) circle (10pt);   
                   \draw (2,0) node {$\encircle{2}$}; 
                           \draw (0,0) node {$\encircle{1}$}; 
                                \draw (4,0) node {$\encircle{3}$}; 
                           \draw (6,0) node {$\encircle{4}$}; 
         \path[draw,use Hobby shortcut,closed=true]
 (-0.5,0)..(-.25,.3)..(1.8,-0.65)..(4.2,-0.65)..(6.25,.3)..(6.5,0)..(5,-1)..(3,-1.1)..(1,-1)..(-0.5,0);
   \draw[rounded corners=15pt]
  (2-.5,0-.6) rectangle ++(3,1.2);
   \end{tikzpicture}
      $$
$$\scalefont{0.8}
 \begin{tikzpicture}[scale=0.8]
    \foreach \x in {0,2,4,6}
        \fill[white](\x,0) circle (10pt);   
                   \draw (2,0) node {$\encircle{2}$}; 
                           \draw (0,0) node {$\encircle{1}$}; 
                                \draw (4,0) node {$\encircle{3}$}; 
                           \draw (6,0) node {$\encircle{4}$}; 
         \path[draw,use Hobby shortcut,closed=true]
 (-0.5,0)..(-.25,.3)..(2,-0.65)..(4.25,.3)..(4.5,0)..(3,-1)..(2,-1.1)..(1,-1)..(-0.5,0);

  \path[draw,use Hobby shortcut,closed=true]
 (-0.5+2,0)..(-.25+2,-.3)..(2+2,0.65)..(4.25+2,-.3)..(4.5+2,-0)..(3+2,1)..(2+2,1.1)..(1+2,1)..(-0.5+2,0);
    \end{tikzpicture}
\qquad\qquad\qquad  
   \begin{tikzpicture}[scale=0.8]
    \foreach \x in {0,2,4,6}
        \fill[white](\x,3.5) circle (10pt);   
                   \draw (2,3.5) node {$\encircle{2}$}; 
                           \draw (0,3.5) node {$\encircle{1}$}; 
                                \draw (4,3.5) node {$\encircle{3}$}; 
                           \draw (6,3.5) node {$\encircle{4}$}; 
     \draw[rounded corners=15pt]
  (-0.5,2.75) rectangle ++(7,1.5);
       \end{tikzpicture}
       $$
\end{eg}

\begin{prop}
Given $r \in \mathbb{N}$, the depth radical ${\sf DR}(\Fdelta)$  is a $\Pdelta$-submodule of 
$\Fdelta$.  
\end{prop}

\begin{proof}
 Clearly the generators $\sgen _{i,i+1}$ for $1\leq i <r$ preserve the space ${\sf DR}(\Fdelta)$
 as both conditions of \cref{depthradical} are invariant under
the permutation action.
By \cref{p1}, the generator   $\pgen _1$ acts on a given  $[\Lambda,\Lambda']$  either 
by scalar multiplication, or by removing an edge from $\Lambda$ {\em at the expense} of introducing a singleton into $\Lambda'$.  Therefore the generator  $\pgen _1$ preserves the submodule by $(ii)$ of \cref{depthradical}.  
By \cref{p12} the generator   $\pgen _{1,2}$ acts on a given  $[\Lambda,\Lambda']$  either 
trivially or by introducing an edge in $\Lambda$.   Therefore the generator  $\pgen _{1,2}$ preserves the submodule by $(i)$ of \cref{depthradical}.    
\end{proof}

\begin{defn}\label{dq}
Define the {\em depth quotient}  ${\sf DQ}(\Fdelta) $ of  $\Fdelta$ to be the quotient $$ {\sf DQ}(\Fdelta) = \Fdelta/{\sf DR}(\Fdelta) $$ 
spanned by 
the   diagrams 
$ [\{\{1\}, \{2\}, \dots, \{r\}\}, \Lambda']$ where $\Lambda'$ contains no singleton blocks.
  \end{defn}

Recall that for $\delta_1\delta_2 \neq 0$  the idempotent $\egen _1=\frac1{\delta_1\delta_2}\pgen _1 \in \Pdelta$.
 By the general theory of idempotent truncation (see for example \cite[Section 6.2]{green}) and  \cref{bob,bob2} we obtain the following.  

\begin{prop}\label{whynamed}
%
%
For $r \ge 2$,
$$
{\sf DR}(\Fdelta) \egen _1 \Pdelta
={\sf DR}(\Fdelta),$$
$${\sf DQ} (\Fdelta)\egen _1 =0,$$
and $${\sf DR}(\Fdelta) \egen _1 \cong  \mathbb{F}^{r-1}(\delta_1, \delta_2) $$
as an $ \egen _1 \Pdelta  \egen _1 \cong P_{r-1}(\delta_1\delta_2)$-module. 
%
When $r=1$, $$\mathbb{F}^{1}(\delta_1, \delta_2)\cong \Delta_{1, \delta_1\delta_2}(\emptyset).$$
\end{prop}

\begin{proof}
 We consider the first  statement. We let $[\Lambda,\Lambda']$ be an arbitrary basis element of ${\sf DR}(\Fdelta)$.  We shall write $[\Lambda,\Lambda']$ in the form
 $$
 [\Lambda,\Lambda'] =  [\overline{\Lambda},\overline{\Lambda'}] \egen _1 d
 $$
 for some $ [\overline{\Lambda},\overline{\Lambda'}] \in {\sf DR}(\Fdelta)$ and some partition diagram $d\in \Pdelta$
  and hence deduce the result.  First, suppose that $ \Lambda'  $ contains a singleton block  
 $ \{i\} $ for $1\leq i \leq r$.  In this case we set 
 $$
 [\overline{\Lambda},\overline{\Lambda'}]
 =
  [\Lambda,\Lambda'] \sgen _{1,i},$$
  where $\sgen _{1,i}=\sgen _{i-1,i}\cdots \sgen _{2,3} \sgen _{1,2}  \sgen _{2,3} \cdots \sgen _{i-1,i}$.  
We find
$$
[\Lambda,\Lambda'] = [\overline{\Lambda},\overline{\Lambda'}] \egen _1 \sgen _{1,i}
$$as required.   Now suppose that $  {\Lambda'} $ contains no singleton block; by \cref{depthradical} we deduce that $\Lambda$ contains a non-singleton block.  
In other words, we suppose that there exist distinct $j,k \in \{1, \ldots, r\}$ with $j \sim_\Lambda k$.
In this case we set 
 $$
 [\overline{\Lambda},\overline{\Lambda'}]
 =
[\Lambda,\Lambda'] \sgen _{1,j} \sgen _{2,k}, $$
where $\sgen _{2,k}=\sgen _{1,2}\sgen _{1,k}\sgen _{1,2}$.
We observe that 
$$
[\Lambda,\Lambda'] = [\overline{\Lambda},\overline{\Lambda'}]  \egen _1 (\pgen _{1,2} \sgen _{1,j} \sgen _{2,k} )
$$ as required.  The first statement follows.

 We now consider the second and third statements.  Let   $[\Lambda,\Lambda']$ be a  basis element of $\Fdelta$ and consider $[\Lambda,\Lambda']  \egen _1$  using \cref{p1}.
In all three cases the resulting outer partition contains a singleton block and therefore $[\Lambda,\Lambda']  \egen _1 \in {\sf DR}(\Fdelta)$.
Therefore the second statement holds.
Considering only $[\Lambda,\Lambda'] \in{\sf DR}(\Fdelta)$, we see that all possible $[\Gamma,\Gamma'] $ with a singleton part $\{1\}$ in both $ \Gamma$ and $\Gamma'$ can occur as  $[\Lambda,\Lambda']  \egen _1$, thus the third statement holds.

%
Finally, it is clear that $\mathbb{F}^{1}(\delta_1, \delta_2)$ is 1-dimensional and $\pgen _1$ acts by scalar multiplication by $\delta_1\delta_2$ as in $\Delta_{1, \delta_1\delta_2}(\emptyset)$.
   \end{proof}

\begin{cor}\label{cor:compfactors+DR}
In the case where $\Pdelta$ is semisimple, we have the following equality of composition multiplicities:
$$
\left[\Fdelta  \, : \, L_{r,\delta_1\delta_2}(\lambda)  \right]_{\Pdelta }
=
\begin{cases} \left[  {\sf DQ}(\Fdelta) \, : \, L_{r,\delta_1\delta_2}(\lambda)  \right]_{\Pdelta }  & \textrm{if } |\lambda|=r,\\
\left[ \mathbb{F}^{r-1}(\delta_1, \delta_2) \, : \, L_{r-1,\delta_1\delta_2}(\lambda)  \right]_{P_{r-1}(\delta_1\delta_2) }  & \textrm{if } |\lambda|<r.\end{cases}$$
\end{cor}
\begin{proof}
This follows from \cref{whynamed} and the construction of simple modules of the partition algebra in \cref{2.2}. 
\end{proof}

We now describe these composition multiplicities in the semisimple case.  If $\lambda$ is a partition of $r$ then, by \cref{annihilate2}, the simple $\Pdelta$-module  $ L_{r,\delta_1\delta_2}(\lambda)$ is  the (inflation of the) Specht module $ {\sf S}(\lambda)$. Therefore, for $|\lambda|=r$,
\begin{equation} \label{eqn:S_sym}
\left[  {\sf DQ}(\Fdelta) \, : \, L_{r,\delta_1\delta_2}(\lambda)  \right]_{\Pdelta }=
\left[  {\sf DQ}(\Fdelta) \, : \, {\sf S}(\lambda)  \right]_{\CC \Sym _r }.
\end{equation}
Now, as a $\CC\Sym _r$-module, ${\sf DQ}(\Fdelta)$ is a permutation module and its decomposition into transitive permutation modules is readily seen (from \cref{dq}) to be
$${\sf DQ}(\Fdelta)=\bigoplus_{\mu \in \mathscr{P}_1(r)} [\Lambda_{(1^r)}, \Lambda_{\mu}] \CC \Sym _r,$$
where, recall,  $\mathscr{P}_1(r)$ denotes the set of  partitions of $r$ which have no part equal to $1$, and, for $\mu$ a partition of $r$ we define a corresponding set-partition $\Lambda_{\mu}=\{\{1, 2, \ldots, \mu_1\}, \{\mu_1+1, \ldots, \mu_2\}, \ldots \}$. Therefore, the  $\CC\Sym _r$-module
\begin{equation}\label{eqn:S_dec}
{\sf DQ}(\Fdelta)=\bigoplus_{\mu \in \mathscr{P}_1(r)} \ind_{\mathrm{Stab}(\Lambda_\mu)}^{\Sym _r} \CC.
\end{equation}
The  groups $\mathrm{Stab}(\Lambda_\mu)$ appearing here are direct products of wreath products of symmetric groups as in \cref{genwreathprodsubgroup}.
\begin{thm}\label{theoremtheorem}
Suppose that $\Pdelta$ is semisimple and $\lambda$ is a partition of $r$. Then
$$
\left[  {\sf DQ}(\Fdelta) \, : \, L_{r,\delta_1\delta_2}(\lambda)  \right]_{\Pdelta }
= \sum_{\mu \in \mathscr{P}_1(r)}\left[   \ind_{\mathrm{Stab}(\Lambda_\mu)}^{\Sym _r} \CC \, : \, {\sf S}(\lambda)  \right]_{\CC \Sym _r }.
$$

\end{thm}

\begin{proof}
This follows from \cref{eqn:S_dec} and \cref{eqn:S_sym}.
\end{proof}


\begin{eg} \label{exampleee2}
We continue the example of $\mathbb{F}^4(\delta_1, \delta_2)$ 
from \cref{exampleee}
 in the  case where $P_r(\delta_1 \delta_2)$ is semisimple.
 The first three diagrams shown in that example belong to the same orbit; the first is $\Lambda_{(2,2)}$, and the stabiliser of any
one of these diagrams is isomorphic to   $\mathrm{Stab}(\Lambda_{(2,2)})=\Sym _2\wr \Sym _2$.
The final diagram is $\Lambda_{(4)}$ and its stabiliser  is $\mathrm{Stab}(\Lambda_{(4)})=\Sym _4$.  
If $\delta_1\delta_2 \notin\{0,1,2,3,4,5,6\}$ then $P_r(\delta_1 \delta_2)$ is semisimple by \cref{ss}.
Since $\ind_{\Sym _2\wr \Sym _2}^{\Sym _4}(\CC)\cong {\sf S}(4)\oplus{\sf S}(2^2)$, 
$$ {\sf DQ}(\mathbb{F}^4(\delta_1, \delta_2))\cong  L_{4,\delta_1\delta_2}(4)\oplus L_{4,\delta_1\delta_2}(4)\oplus L_{4,\delta_1\delta_2}(2^2).$$
The rest of the decomposition of $\mathbb{F}^4(\delta_1, \delta_2)$, the decomposition of its depth radical, is obtained by \cref{cor:compfactors+DR}.
 Accordingly, we must decompose the $P_3(\delta_1\delta_2)$-module  $\mathbb{F}^3(\delta_1, \delta_2)$. As $ {\sf DQ}(\mathbb{F}^3(\delta_1, \delta_2)$ is 1-dimensional, with basis vector
$$
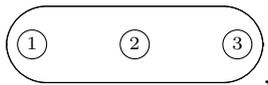
\scalefont{0.7}\begin{tikzpicture}[scale=0.675]
\path(2,-6) coordinate (origin);   \path   (origin)--++(-2.5,-0.75) coordinate (origin2);
       \draw[fill=white,rounded corners=15pt]
  (origin2)  rectangle ++(5,1.5);

    \foreach \x in {-2,0,2}
        \fill[white](origin)--++(\x,0) circle (10pt);   
                   \path (origin)--++(0,0) node {$\encircle{2}$}; 
                           \path (origin)--++(-2,0) node {$\encircle{1}$}; 
                                \path (origin)--++(2,0) node {$\encircle{3}$}; 
    \path   (origin)--++(-2.5,-0.75) coordinate (origin2);
       \draw[rounded corners=15pt]
  (origin2)  rectangle ++(5,1.5);
     
 \end{tikzpicture} 
$$
 it follows that $ {\sf DQ}(\mathbb{F}^3(\delta_1, \delta_2))$   is isomorphic to  $  L_{3,\delta_1\delta_2}(3)$ and so contributes a summand $  L_{4,\delta_1\delta_2}(3)$ to $\mathbb{F}^4(\delta_1, \delta_2)$.
 To decompose ${\sf DR} (\mathbb{F}^3(\delta_1, \delta_2))$, we  must decompose the  $P_2(\delta_1\delta_2)$-module  $\mathbb{F}^2(\delta_1, \delta_2)$. By the previous argument, $ {\sf DQ}(\mathbb{F}^2(\delta_1, \delta_2)$ is isomorphic to 
  $  L_{2,\delta_1\delta_2}(2)$, and the decomposition of  ${\sf DR} (\mathbb{F}^2(\delta_1, \delta_2))$ is governed by that of 
$\mathbb{F}^1(\delta_1, \delta_2) \cong L_{1,\delta_1\delta_2}(\emptyset)$.  Putting all this together shows
$${\sf DR} (\mathbb{F}^4(\delta_1, \delta_2)) \cong  L_{4,\delta_1\delta_2}(3)\oplus  L_{4,\delta_1\delta_2}(2)\oplus L_{4,\delta_1\delta_2}(\emptyset).$$
Combining the decompositions of ${\sf DR} (\mathbb{F}^4(\delta_1, \delta_2))$ and ${\sf DQR} (\mathbb{F}^4(\delta_1, \delta_2))$ obtained above shows that
$$\mathbb{F}^4(\delta_1, \delta_2)  \cong  2    L_{4,\delta_1\delta_2}(4) \oplus L_{4,\delta_1\delta_2}(2^2) \oplus L_{4,\delta_1\delta_2}(3)\oplus  L_{4,\delta_1\delta_2}(2)\oplus L_{4,\delta_1\delta_2}(\emptyset).$$

\end{eg}

  \begin{rmk}
  Alternatively, one can note that we have an injective
 $\CC$-linear map $\varphi:  \mathbb{F}^{r-1}(\delta_1, \delta_2) \to \Fdelta$  given by 
 $\varphi( \{S_1,S_2,\dots,S_p\},
 \{
 \Sigma_1,\Sigma_2,\dots \Sigma_q 
 \}
 )= (  
  \{
  S_1,S_2,\dots,S_p, \{r\}
  \},
 \{\Sigma_1,\Sigma_2,\dots \Sigma_q, \{r\}\})
$.  
 The image of this map generates the submodule ${\sf DR}(\Fdelta) $.  
The form of the map $\varphi$   is how we first arrived at condition $(ii)$ of \cref{depthradical};
 condition $(i)$  was then deduced by considering the submodule generated by  the image.
\end{rmk}

\section{Consequences for plethysm coefficients}\label{sec9}
For the final section we specialise $\delta_1$ and $\delta_2$ to be $m$ and $n$ respectively.
 Combining \cref{cor:compfactors+DR} and \cref{theoremtheorem} with \cref{schurfunctor}, we obtain a formula for certain plethysm coefficients in terms of  smaller generalised plethysm coefficients (as defined in \cref{genplethy}).

\begin{thm} \label{thm:inductive_description}
Let $m,n\in \mathbb{N}$ and let $\lambda_{[mn]}$ be a partition of $mn$ with with $m,n\geq |\lambda|$. Then
$$ p(   ( n), (m), \lambda_{[mn]})= \sum_{\boldsymbol\mu \in  \mathscr{P}_1(|\lambda|)} p_{\boldsymbol\mu}( \lambda).$$ \end{thm}
\begin{proof}
Let $r = |\lambda|$. Under the hypotheses, $mn\ge r^2>2r-2$ and so $P_r(mn)$ is semisimple. 
 Schur--Weyl duality and \cref{schurfunctor2} show that 
$$ p(   ( n), (m), \lambda_{[mn]})=\left[ \Fmn \, : \, L_{r, mn}(\lambda))  \right]_{\Pmn}.$$
Applying \cref{theoremtheorem} then  gives
$$ p(   ( n), (m), \lambda_{[mn]})=
\sum_{\boldsymbol\mu \in  \mathscr{P}_1(|\lambda|)}\left[  \ind_{\mathrm{Stab}(\Lambda_{\boldsymbol\mu)}}^{\Sym _{|\lambda|}} \CC \, : \, {\sf S}(\lambda)  \right]_{\CC \Sym _{|\lambda|}}= \sum_{\boldsymbol\mu \in  \mathscr{P}_1(|\lambda|)} p_{\boldsymbol\mu}( \lambda).$$
\end{proof}
\begin{eg}
Continuing \cref{exampleee} and \cref{exampleee2}, when $r=4$ and $m,n \ge 4$ we obtain the following plethysm coefficients:
\begin{align*}
 p ((n),(m),(mn-4,4)) &= 2
\\
 p ((n),(m),(mn-4,2^2))&= 1
 \\
 p ((n),(m),(mn-3,3))&= 1\\ p ((n),(m),(mn-2,2))&= 1\\ p ((n),(m),(mn)) &=1 
 \end{align*}
and the coefficients $p ((n),(m) ,\alpha)=0$ for 
all other partitions $\alpha$  of depth at most $ 4$.  
\end{eg}

The following corollaries are immediate from \cref{thm:inductive_description}.

\begin{cor}\label{cor:stability}
Let $\lambda$ be an arbitrary partition.  
The double sequence 
$$\{p(   ( n), (m), \lambda_{[mn]})\}_{m,n\in\mathbb{N}}$$
stabilises 
for all $m,n \geq |\lambda|$;
 we denote the stable values by $\overline{p}_{\infty,\lambda}$.  
In other words  $$\overline{p}_{\infty,\lambda}=p(   ( n), (m), \lambda_{[mn]}) $$
for all $m,n
\geq |\lambda|$.  
 \end{cor}
 
%
%

\begin{prop}\label{sharp}
We have that 
$$
\overline{p}_{\infty,(r)}=
p((n),(m),(mn-r,r))=
p ((r),(r),(r^2-r,r))=
 |  \mathscr{P}_1(r)|.
$$for $n,m\geq r$.  
However, for $r>2$,
 $$p(	  (r-1), (r) , (r^2-2r,r))=p ( (r),(r-1) , (r^2-2r,r))=|  \mathscr{P}_1(r)|-1.$$
In particular, the stability in \cref{cor:stability} is sharp for   $\lambda $ an arbitrary partition of~$r$.  
\end{prop}

\begin{proof}
The first statement follows immediately from 
\cref{thm:inductive_description} 
and the fact that 
${\sf S}((r))$ occurs as a summand once in each transitive $\C{\Sym _r}$-permutation module $\ind_{\mathrm{Stab}(\Lambda_\mu)}^{\Sym _r} \CC$.
For the second part, recall that the Cayley-Sylvester formula provides the plethysm coefficients for two-part partitions:
$$ p ((n),(m),(mn-r,r))=|  \mathscr{P}_{m \times n} (r)|- |  \mathscr{P}_{m \times n} (r-1)|,$$
where $   \mathscr{P}_{m \times n} (r)$ equals the set of all those partitions of $r$ whose Young diagrams fit inside the $m \times n$ rectangle. Taking $m=r$ and $n=r-1$ (and the calculation is identical for $m=r-1$, $n=r$),
\begin{eqnarray*}
p(({r-1}),(r),(r^2-2r,r))&=&|  \mathscr{P}_{r \times (r-1)} (r)|- |  \mathscr{P}_{r \times (r-1)} (r-1)|\\
&=& |  \mathscr{P}(r)|-1 -|  \mathscr{P}(r-1)|\\
&=& |  \mathscr{P}_1(r)|-1,
\end{eqnarray*}
since adding a part of size~1 provides a bijection $\mathscr{P}(r-1)\to \mathscr{P}(r) \setminus \mathscr{P}_1(r)$.  

\end{proof}

\begin{eg}
The 
non-zero
stable plethysm coefficients $\overline{p}_{\infty,\lambda}$ for $\lambda\vdash 8$ are as follows:
$$
\overline{p}_{\infty,(8)}=7,
\quad 
\overline{p}_{\infty,(7,1)}=4,
\quad 
\overline{p}_{\infty,(6,2)}=8,
\quad 
\overline{p}_{\infty,(5,3)}=3,
\quad 
\overline{p}_{\infty,(5,2,1)}=2,
$$
$$  \overline{p}_{\infty,(4^2)}=4,\quad 
\overline{p}_{\infty,(4,3,1)}=1,\quad 
\overline{p}_{\infty,(4,2^2)}=3,
\quad 
\overline{p}_{\infty,(2^4)}=1.
$$
To see this, we note that the elements of $\mathscr{P}_1(8)$ are 
$(8)$, $(6,2)$, $(5,3)$, $(4,4)$, $(4,2^2)$, $(3^2,2)$
and $(2^4)$.  
The required products of (smaller) plethysm and Littlewood-Richardson coefficients
can be calculated
 by hand or 
 using SAGE (whereas the coefficients  $p((8),(8),\lambda)$ seem to be beyond  SAGE).  
The
decompositions of the relevant transitive permutation modules are as follows:
$$
\ind^{\Sym _{8}}_{\Sym _{8}}(\Bbbk) 
=  {\sf S}(8) ,
\qquad  \ind^{\Sym _{8}}_{\Sym _6 \times\Sym _2 }(\Bbbk) = 
  {\sf S}(6,2) 
\oplus {\sf S}(7,1) 
\oplus {\sf S}(8) ,
$$
$$
 \ind^{\Sym _{8}}_{\Sym _5 \times \Sym _3}(\Bbbk) 
=  {\sf S}(5,3) 
\oplus {\sf S}(6,2) \oplus {\sf S}(7,1)  \oplus {\sf S}(8) ,
 \qquad  \ind^{\Sym _{8}}_{\Sym _4 \wr \Sym _2}(\Bbbk) 
=
 {\sf S}(4^2) 
 \oplus {\sf S}(6,2) 
\oplus {\sf S}(8), 
$$   
$$
 \ind^{\Sym _{8}}_{\Sym _{4} \times \Sym _2 \wr \Sym _2}(\Bbbk) 
=
{\sf S}(4,2^2)
\oplus {\sf S}(4^2)
\oplus {\sf S}(5,2,1)
\oplus {\sf S}(5,3)
\oplus 2{\sf S}(6,2) 
\oplus {\sf S}(7,1)
\oplus {\sf S}(8),
 $$
 $$
 \ind^{\Sym _{8}}_{\Sym _3\wr \Sym _2 \times \Sym _2}(\Bbbk) 
=
{\sf S}(4,2^2	)\oplus
{\sf S}(4,3,1	)\oplus{\sf S}(4^2	)\oplus
{\sf S}(	5,2,1)\oplus{\sf S}(	5,3)\oplus
2{\sf S}(6,2	)\oplus
{\sf S}(7,1	)\oplus
{\sf S}(8	), 
 $$
 $$
 \ind^{\Sym _{8}}_{\Sym _2\wr \Sym _4}(\Bbbk) 
=
 {\sf S}(	2^4)\oplus
  {\sf S}(4,2^2	)\oplus {\sf S}(	4^2)\oplus {\sf S}(6,2	)\oplus {\sf S}(8). 
 $$
One can deduce  that  Foulkes' conjecture 
  is satisfied at the partition $\lambda_{[mn]}$ for all  pairs  $m,n$ of integers which are both greater than or equal to $ 8$.  For example,  
$$
\overline{p}_{\infty,(6,2)}=
p((8),(8),(56,6,2))=
p((9),(8), (64,6,2))
=
p((8),(9), (64,6,2))=8.$$
We  also deduce the strengthened   Foulkes' conjecture 
  is satisfied for all  quadruples  of integers which are each greater than or equal to $ 8$.  For example,  
 $$
 p((240),(8), (1912	,6,2)	)
  =p((40),(48), (1912	,6,2)	)
  =
   p((16),(120), (1912	,6,2)	)
=8.   $$
\end{eg}

\begin{eg}
We find $   p( (10),(10),(90, 4^2,2))=6.$   This can be calculated as follows: $$
[\ind^{\Sym _{10}}_{\Sym _4 \wr\Sym _2 \times \Sym _2}(\Bbbk) :
{\sf S} (4^2,2)]=1,\quad
[\ind^{\Sym _{10}}_{\Sym _4 \times \Sym _3 \wr\Sym _2 }(\Bbbk) :
{\sf S} (4^2,2)]=1,\quad
[\ind^{\Sym _{10}}_{\Sym _4 \times \Sym _2 \wr\Sym _3 }(\Bbbk) :
{\sf S} (4^2,2)]=1,$$
$$
[\ind^{\Sym _{10}}_{\Sym _3\wr\Sym _2 \times \Sym _2 \wr \Sym _2}(\Bbbk) :
{\sf S} (4^2,2)]=2, \quad
[\ind^{\Sym _{10}}_{\Sym _2 \wr \Sym _5}(\Bbbk) :
{\sf S} (4^2,2)]=1 
$$
with $$[\ind^{\Sym _{10}}_{\mathrm{Stab}(\Lambda_{\mu})}(\Bbbk) :
{\sf S} (4^2,2)]=0 $$ for all other  $\mu \in \mathscr{P}_1(10)$.  
\end{eg}

Our results provide an elementary new proof of  Weintraub's Conjecture \cite{MR1037395} for stable plethysm coefficients. Recall that the conjecture, recently proven in \cite{BCI}, states that if $m$ is even and $\lambda=(\lambda_1,\lambda_2,\dots, \lambda_\ell)$ is an even partition   (that is $\lambda_1,\lambda_2,\dots, \lambda_\ell\in 2\NN$) then the plethysm coefficient $p((n), (m), \lambda)$ is non-zero.

\begin{cor}[Stable  Weintraub's conjecture]\label{cor:weintraub}
For $\lambda$  an even partition,  we have that
$$
\overline{p}_{\infty,\lambda} 
>0.
$$
\end{cor}
\begin{proof}
Let $\lambda=(a_1^{b_1},a_2^{b_2},\dots, a_\ell^{b_\ell})$ be an even  partition, 
and pick $m,n\geq |\lambda|$. We use the formula for $
p(   ( n), (m), \lambda_{[mn]})$ in \cref{thm:inductive_description}. Since $\lambda$ is even, $\lambda \in \mathscr{P}_1(|\lambda|)$. The contribution to the sum from taking ${ \boldsymbol \mu}=\lambda$  is~1 since
 $p ( (b_i),(a_i), ({a_i}^{b_i}))=1$
  for even $a_i$ by \cite[Theorem 2.6]{PW} and, by the Littlewood--Richardson rule, 
$p_{ \boldsymbol \mu}(\lambda )=1$ for $\boldsymbol \mu =(a_1^{b_1},\dots, a_\ell^{b_\ell})=\lambda$. Therefore the stable plethysm coefficient $\overline{p}_{\infty,\lambda}=p(   ( n), (m), \lambda_{[mn]})$ is strictly positive.
 \end{proof}





\begin{thebibliography}{99}

\bibitem{BDE}
C.~Bowman, M.~{De Visscher}, and J.~Enyang, \emph{The co-Pieri rule for stable Kronecker coefficients}, 
J. Combin. Theory Ser. A \textbf{177} (2021), Paper No. 105297, 71 pp. 


\bibitem{BDO15}
C.~Bowman, M.~{De Visscher}, and R.~Orellana, \emph{The partition algebra and
  the {K}ronecker coefficients}, Trans. Amer. Math. Soc. \textbf{367} (2015),
  no.~5, 3647--3667.

\bibitem{brion}
M.~Brion, \emph{Stable properties of plethysm: on two conjectures of
  {F}oulkes}, Manuscripta Math. \textbf{80} (1993), 347--371.

\bibitem{BCI}
Peter B\"{u}rgisser, Matthias Christandl, and Christian Ikenmeyer, \emph{Even
  partitions in plethysms}, J. Algebra \textbf{328} (2011), 322--329.
  \MR{2745569}

\bibitem{MR1190119}
C.~Carr\'e and J.-Y. Thibon, \emph{Plethysm and vertex operators}, Adv. in
  Appl. Math. \textbf{13} (1992), no.~4, 390--403.

\bibitem{MR3574536}
M.~Cheung, C.~Ikenmeyer, and S.~Mkrtchyan, \emph{Symmetrizing tableaux and the
  5th case of the {F}oulkes conjecture}, J. Symbolic Comput. \textbf{80}
  (2017), no.~part 3, 833--843.

\bibitem{dent}
S.~Dent and J.~Siemons, \emph{On a conjecture of {F}oulkes}, J. Algebra
  \textbf{226} (2000), no.~1, 236--249.

\bibitem{MR0037276}
H.~O. Foulkes, \emph{Concomitants of the quintic and sextic up to degree four
  in the coefficients of the ground form}, J. London Math. Soc. \textbf{25}
  (1950), 205--209.

\bibitem{green}
James~A. Green, \emph{Polynomial representations of {${\rm GL}_{n}$}}, Lecture
  Notes in Mathematics, vol. 830, Springer-Verlag, Berlin-New York, 1980.
  \MR{606556}

\bibitem{HalvRam}
T.~Halverson and A.~Ram, \emph{{Partition algebras}}, {European J. Combin.}
  \textbf{{26}} ({2005}), no.~{6}, {869--921}.

\bibitem{james}
G.~D. James, \emph{The representation theory of the symmetric groups}, Lecture
  Notes in Mathematics 682, Springer, 1978.

\bibitem{jk}
G.~D. James and A.~Kerber, \emph{The representation theory of the symmetric
  group}, Encyclopedia of Mathematics and its Applications, vol.~16,
  Addison-Wesley, 1981.

\bibitem{jones}
V.~F.~R. Jones, \emph{The {P}otts model and the symmetric group}, In:
  Subfactors: Proceedings of the Tanaguchi Symposium on Operator Algebras
  (Kyuzeso, 1993) (NJ), World Sci. Publishing River Edge, 1994, pp.~259--267.

\bibitem{MR1651092}
L.~Manivel, \emph{Gaussian maps and plethysm}, Algebraic geometry ({C}atania,
  1993/{B}arcelona, 1994), Lecture Notes in Pure and Appl. Math., vol. 200,
  Dekker, New York, 1998, pp.~91--117.

\bibitem{marbook}
P.~P. Martin, \emph{Potts models and related problems in statistical
  mechanics}, Series on Advances in Statistical Mechanics, 5, World Scientific
  Publishing Co., Inc., Teaneck, NJ, 1991.

\bibitem{mar1}
\bysame, \emph{The structure of the partition algebras}, J. Algebra
  \textbf{183} (1996), 319--358.
  
  \bibitem{INSERT}
P. P.   Martin,   A.  Elgamal, 
{\em Ramified partition algebras} 
Math. Z. 246 (2004), no. 3, 473--500. 

\bibitem{msannular}
P.~P. Martin and H.~Saleur, \emph{On an algebraic approach to higher
  dimensional statistical mechanics}, Comm. Math. Phys. \textbf{158} (1993),
  155--190.

\bibitem{mckay}
T.~McKay, \emph{On plethysm conjectures of {S}tanley and {F}oulkes}, J. Algebra
  \textbf{319} (2008), no.~5, 2050--2071.


\bibitem{party}
R. Orellana, F. Saliola, A. Schilling, and  M. Zabrocki
{\em Plethysm and the algebra of uniform block permutations},
arXiv:2112.13909. 
 

\bibitem{PW}
R. Paget and M. Wildon, \emph{Set families and {F}oulkes modules}, J.
  Algebraic Combin. \textbf{34} (2011), no.~3, 525--544.

\bibitem{MR1754784}
R.~P. Stanley, \emph{Positivity problems and conjectures in algebraic
  combinatorics}, Mathematics: frontiers and perspectives, Amer. Math. Soc.,
  Providence, RI, 2000, pp.~295--319.

\bibitem{thrall}
R.~M. Thrall, \emph{On symmetrized {K}ronecker powers and the structure of the
  free {L}ie ring}, Amer. J. Math. \textbf{64} (1942), 371--388.

\bibitem{MR2067621}
R.~Vessenes, \emph{Generalized {F}oulkes' conjecture and tableaux
  construction}, J. Algebra \textbf{277} (2004), no.~2, 579--614.

\bibitem{MR1037395}
S.~Weintraub, \emph{Some observations on plethysms}, J. Algebra \textbf{129}
  (1990), no.~1, 103--114. \MR{1037395}

\end{thebibliography}
 \end{document}